\documentclass[reqno]{amsart}

\makeatletter
	
	\@addtoreset{equation}{section}

	\@addtoreset{figure}{section}
\makeatother

\usepackage{longtable}
\usepackage{booktabs}
\usepackage{xcolor}
\usepackage{siunitx}
\usepackage{amsthm}
\usepackage{amsmath}
\usepackage{amsfonts}
\usepackage{graphicx}
\usepackage{subcaption}
\usepackage{url}
\usepackage{mathtools}
\usepackage{caption}
\usepackage{bm}
\usepackage{booktabs}
\usepackage{here}
\usepackage{enumerate}
\usepackage{float}
\usepackage{a4wide}
\usepackage{comment}
\usepackage{multirow}
\usepackage{hyperref}
\usepackage{thmtools}
\makeatletter
\renewcommand\thmt@autorefsetup{%
  \@xa\def
  \csname\thmt@envname autorefname\@xa\endcsname
  \@xa{\thmt@thmname}%
}
\makeatother
\hypersetup{
	colorlinks=false, 
	bookmarks=true, 
	bookmarksnumbered=true,
	urlbordercolor={0 1 1},
	bookmarkstype=toc
}

\title[]{On the volume conjecture of the colored Jones invariants with arbitrary colors}
\author[]{Shinichiro Kakuta}
\address{Department of Mathematics, School of Fundamental Science and Engineering, Faculty of Science and Engineering, Waseda University, 3-4-1 Ohkubo, Shinjuku-ku, Tokyo, 169-8555, Japan}
\email{s.kakuta@ruri.waseda.jp}

\newtheorem{theorem}{Theorem}[section]
\newtheorem{proposition}[theorem]{Proposition}
\newtheorem{lemma}[theorem]{Lemma}

\newtheorem{conjecture}[theorem]{Conjecture}

\theoremstyle{definition}
\newtheorem{remark}[theorem]{Remark}

\allowdisplaybreaks

\begin{document}
\begin{abstract}
We study the volume conjecture of the colored Jones invariants with sequences of colors corresponding to the deformation of the hyperbolic structure of a link complement. In particular, we investigate certain limits of the colored Jones invariants of the figure-eight knot and the Borromean rings and show that the limits are related to the volumes of hyperbolic cone manifolds whose singular sets are the links.
\end{abstract}
\maketitle
\section{Introduction}
The volume conjecture \cite{Kashaev1997,Murakami-Murakami2001} predicts that the exponential growth rate of quantum invariants at roots of unity is related to the hyperbolic volume of a link complement. Its complexified form connects the real and imaginary parts of an appropriate limit with the hyperbolic volume and the Chern--Simons invariant (see \cite{MurakamiEtAl2002}). By varying the deformation parameter, it is conjectured that the quantum invariants detect the volume of the link complement whose hyperbolic structure is deformed. Indeed, relations of volume conjecture type between the colored Jones invariant at a different value and the volume of a hyperbolic cone manifold have already been discussed in \cite{MurakamiH2004} and \cite{Murakami-Yokota2007} (see also \cite{Gukov2005} and \cite{Gukov-Murakami2008}). Moreover, allowing not only the deformation parameter but also the colors to vary leads to a volume conjecture proposed by Chen--Yang \cite{Chen-Yang2018} and Murakami \cite{MurakamiJ2021} as follows (see also \cite{MurakamiJ2014} and \cite{MurakamiJ2018}).

Let $L$ be a hyperbolic link in $S^3$ with $\ell$ components $L_1,\ldots,L_\ell$, let $\alpha_1,\ldots,\alpha_\ell$ be constants in $[0,\pi]$, and let $M_{\alpha_1,\ldots,\alpha_\ell}(L)$ be the hyperbolic cone manifold with singular set $L$ whose cone angle around $L_i$ is $\alpha_i$. Note that we consider $\alpha_i$ only when $M_{\alpha_1,\ldots,\alpha_\ell}(L)$ is hyperbolic. Let $r$ be an odd integer greater than or equal to three and let $V_{j_1,\ldots,j_\ell}^{(r)}(L)$ be the colored Jones invariant of $L$ whose components $L_1,\ldots,L_\ell$ are colored by weights $j_1,\ldots,j_\ell$ in $\{j/2\,|\,j\in\mathbb{Z},0\leq j\leq r-2\}$ respectively, where the parameter is $q=\exp(4\pi\sqrt{-1}/r)$. Moreover, we put $r=2n+1$. Then the following may hold.
\begin{conjecture}[{\cite[Conjecture 4]{MurakamiJ2021}}]\label{mainconjecture}
For each $i=1,\ldots,\ell$, let $j_i$ be weights such that
\[\left|8\pi\lim_{n\to\infty}\frac{j_i}{2n+1}-2\pi\right|=\alpha_i,\]
then 
\[\lim_{n\to\infty}\frac{4\pi}{2n+1}\log|V_{j_1,\ldots,j_\ell}^{(2n+1)}(L)|=\mathrm{Vol}(M_{\alpha_1,\ldots,\alpha_\ell}(L)),\]
where $\mathrm{Vol}(M_{\alpha_1,\ldots,\alpha_\ell}(L))$ is the hyperbolic volume of $M_{\alpha_1,\ldots,\alpha_\ell}(L)$.
\end{conjecture}

In this paper, we investigate such a limit for the figure-eight knot $E$ and the Borromean rings $B$, and show that the limits coincide with the hyperbolic volumes under explicit assumptions. 

Let $\Phi_E(z)$ be the potential function of $V_{j}^{(r)}(E)$. If we consider the imaginary part of $\Phi_E(x)$ as a real-valued function on a certain interval, then it has the maximal value corresponding to the volume. Then we have the following.

\begin{restatable}{theorem}{mainthmF}\label{MainTheorem1}
Conjecture \ref{mainconjecture} is true for the figure-eight knot $E$ and for cone angles $\alpha$ such that $0\leq\alpha<\alpha_0$, where
\[\alpha_0=\sup\left\{\alpha\in\left[0,\frac{2\pi}{3}\right)\,\middle|\,2\,\mathrm{Im}\left(\Phi_E\left(\frac{\alpha}{2}\right)\right)<\mathrm{Vol}(M_\alpha(E))\right\}=1.7647826174\ldots.\]
\end{restatable}
\begin{remark}
Conjecture \ref{mainconjecture} for the figure-eight knot is experimentally checked in \cite{MurakamiJ2021}.
\end{remark}
Cho and Murakami \cite{Cho-Murakami2009} showed that a suitable limit of the colored Alexander invariant which is similar to the colored Jones invariant detects the volume of the hyperbolic orbifold whose underlying space is the knot complement for the figure-eight knot for angles $\alpha=\frac{2\pi}{n}$ ($n\geq5$); these angles are contained in the range $0<\alpha\leq\frac{2\pi}{5}=1.2566370614\ldots$\,. Wong and Yang \cite{Wong-YangII2020} showed that a similar limit of the relative Reshetikhin-Turaev invariant which is proportional to the colored Jones invariant gives the volume of the hyperbolic cone manifold along the figure-eight knot under the condition $\mathrm{Vol}(M_\alpha(E))>\frac{\mathrm{Vol}(S^3\backslash E)}{2}$ which implies $0\leq\alpha<1.2002328742\ldots$\,. The result of Theorem \ref{MainTheorem1} strictly contains both previously known cone angle ranges.

Let $\Phi_B(z)$ be the potential function of $V_{j_1,j_2,j_3}^{(r)}(B)$. The imaginary part of $\Phi_B(x)$ has a distinguished extremal value equal to one half of the corresponding hyperbolic volume. The following is also one of our results.
\begin{restatable}{theorem}{mainthmS}\label{MainTheorem2}
Conjecture \ref{mainconjecture} is true for the Borromean rings $B$ and for cone angles $\alpha_1,\alpha_2,\alpha_3$ such that $(\alpha_1,\alpha_2,\alpha_3)$ belongs to $\Omega_0$, where
\[\Omega_0=\left\{(\alpha_1,\alpha_2,\alpha_3)\in[0,\pi)^3\,\middle|\,2\,\mathrm{Im}\left(\Phi_B\left(\frac{\min\{\alpha_1,\alpha_2,\alpha_3\}}{2}\right)\right)<\mathrm{Vol}(M_{\alpha_1,\alpha_2,\alpha_3}(B))\right\}.\]
\end{restatable}

\begin{remark}
The region of the cone angles satisfying the conditions of Theorem \ref{MainTheorem2} is numerically visualised in Section \ref{SecB}.
\end{remark}

\begin{remark}
We checked Conjecture \ref{mainconjecture} for various larger cone angles not covered by our rigorous proofs in the case of $E$ and $B$ through a numerical experiment.
\end{remark}

Our proofs are based on the analysis of the potential functions associated with the colored Jones invariants for $E$ and $B$. Noting that their colored Jones invariants can be regarded as real-valued functions and that the asymptotic behavior of the summands contains the volume formulae for the hyperbolic cone manifolds, we separate the summation into a part where the signs of the summands are alternating and a part where they have the constant sign. On the one hand, the constant sign part of the summation has exponential growth rate bounded below by the hyperbolic volume. On the other hand, following and applying the analytic framework developed by Ohtsuki \cite{Ohtsuki2016}, we show that the sum of all the Fourier coefficients from the potential function is exponentially smaller than the main contribution, and then we partially rewrite the colored Jones invariants via the Poisson summation formula. This shows that the contribution from the main term dominates the asymptotics of the colored Jones invariant.

This paper is organized as follows. In Section \ref{SecE}, we treat the case of the figure-eight knot  $E$ and prove Theorem \ref{MainTheorem1}. We first recall the colored Jones invariant of $E$ and its potential function, and then give our proof of the volume conjecture for a certain range of the cone angle. In Section \ref{SecB}, we study the Borromean rings $B$ and prove Theorem \ref{MainTheorem2} in a manner similar to the case of $E$. The lemmata used in the proofs of the main results are collected in Section \ref{lemmaproofs}. 

\section*{Acknowledgements}
The author is grateful to Professor Jun Murakami for his helpful comments and encouragement.

\section{The case of the figure-eight knot}\label{SecE}
For the figure-eight knot, the colored Jones invariant and its potential function, along with the volume formula of the hyperbolic cone manifold, are straightforward and relatively simple to consider. In particular, most of the discussion can be conducted using elementary calculus. Furthermore, supplementing this with complex and Fourier analysis allows the results to be extended.

\subsection{The colored Jones invariant for $E$}\label{tcjiaipffE}
Let $s$ be the $r$-th root of unity $\exp(2\pi\sqrt{-1}/r)$. We use the following notations:
\[\{n\}=s^n-s^{-n},\quad \{n;k\}=\prod_{j=0}^{k-1}\{n-j\},\quad \{n\}!=\{n;n\}.\]
Note that the value of the deformation parameter $q$ of the colored Jones invariants we consider is $s$ squared.

The colored Jones invariants can be defined by several well-known methods and given in explicit formulae. It is known that the colored Jones invariant for the figure-eight knot $E$ is given by the formula of Habiro \cite{Habiro2000} and L\^e \cite{Le2003}
\begin{align*}
	V_j^{(r)}(E)&=\sum_{k=0}^{\min\{2j,r-(2j+1)-1\}}\frac{1}{\{1\}}\frac{\{2j+1+k\}!}{\{2j-k\}!}.
\end{align*}
Note that $\{r\}=0$ by the definition. Letting $(s^2)_n=(1-s^2)(1-s^4)\cdots(1-s^{2n})$, we obtain $\{n\}!=(-1)^ns^{-n(n+1)/2}(s^2)_n$, and then $V_j^{(r)}(E)$ has the following form:
\begin{align*}
	V_j^{(r)}(E)&=\sum_{k=0}^{k_{\max}}\frac{1}{\{1\}}(-1)^{2k+1}s^{-(2j+1)(2k+1)}\frac{(s^2)_{2j+1+k}}{(s^2)_{2j-k}},
\end{align*}
where $k_{\max}=\min\{2j,r-(2j+1)-1\}$. 

We now use the following asymptotic behaviors as $r\to\infty$. By the assumption in Conjecture \ref{mainconjecture}, we know that the ratio of weight $j$ and integer $r$ converges to a constant for each fixed $\alpha$: $\frac{j}{r}\sim\frac{1}{4}\pm\frac{\alpha}{8\pi}$ $(r\to\infty)$. Thus, the following approximations are valid. It is known that
\begin{equation}
(s^2)_n\sim\exp\left[\frac{r}{2\pi\sqrt{-1}}\left\{-\frac{1}{2}\mathrm{Li}_2(s^{2n})+\frac{\pi^2}{12}\right\}\right]\quad(r\to\infty),\label{dilogsim}
\end{equation}
where $\mathrm{Li}_2(w)$ is the dilogarithm function defined by the integral
\[\mathrm{Li}_2(w)=-\int_0^w\frac{\log(1-u)}{u}du\]
for $w$ between $0$ and $1$ with its analytic continuation to the complex plane cut along $[1,\infty)$ on the real axis. Note that possible values of $n$ in the form $(s^2)_n$ are $2j+1+k$ and $2j-k$. Thus, we have
\begin{align*}
	V_j^{(r)}(E)&\sim\sum_{k=0}^{k_{\max}}\frac{1}{\{1\}}\exp\left[\frac{r}{2\pi\sqrt{-1}}\left\{-2\pi\frac{2\pi k}{r}+4\frac{2\pi j}{r}\frac{2\pi k}{r}\right.\right.\\
	&\hspace{12em}\left.\left.-\frac{1}{2}\mathrm{Li}_2(e^{2\sqrt{-1}\frac{2\pi(2j+1+k)}{r}})+\frac{1}{2}\mathrm{Li}_2(e^{2\sqrt{-1}\frac{2\pi(2j-k)}{r}})\right\}\right]\quad(r\to\infty).
\end{align*}
We regard $\frac{2\pi k}{r}$ as the continuous real variable $x$ for a large $r$.

By extending $x$ to a complex variable $z$, the exponents with $\frac{r}{2\pi\sqrt{-1}}$ of the exponential functions in the calculated summands turn into
\[\Phi_E(\alpha;z)=-\frac{1}{2}\mathrm{Li}_2(e^{2\sqrt{-1}(\alpha/2+z)})+\frac{1}{2}\mathrm{Li}_2(e^{2\sqrt{-1}(\alpha/2-z)})+\alpha z.\] 
We call $\Phi_E(z)=\Phi_E(\alpha;z)$ the potential function of the colored Jones invariant $V_j^{(r)}(E)$ for $E$. Note that there is an ambiguity of the potential function coming from the choice of the power of $-1$ in the summation. Moreover, if the sign of the angle is changed, the formula
\begin{equation}
\mathrm{Li}_2\left(w^{-1}\right)=-\mathrm{Li}_2(w)-\frac{\pi^2}{6}-\frac{1}{2}\left\{\mathrm{Log}\,(-w)\right\}^2\quad(w\in\mathbb{C}\backslash[0,\infty))\label{inversion}
\end{equation}
shows that only the linear term is shifted, where $\mathrm{Log}$ denotes
\[\mathrm{Log}\,w=\log|w|+\sqrt{-1}\arg w,\quad-\pi<\arg w<\pi.\]

\begin{remark}
The potential function $\Phi_E(z)$ has the branch cuts along
\[\left\{\pm\pi n\mp\frac{\alpha}{2}\mp\sqrt{-1}R\,\middle|\,R\geq0\right\}\quad(n\in\mathbb{Z})\]
in the complex plane.
\end{remark}

Let $\Lambda: \mathbb{R}\to\mathbb{R}$ be the Lobachevsky function defined by
\[\Lambda(s)=-\int_0^s\log|2\sin t|dt.\]
Note that $\Lambda(s)$ is an odd and $\pi$-periodic function. If we consider the potential function $\Phi_E(z)$ for $z=x\in\mathbb{R}$, we get 
\begin{align}
	\mathrm{Im}(\Phi_E(x))&=-\Lambda\left(\frac{\alpha}{2}+x\right)+\Lambda\left(\frac{\alpha}{2}-x\right)=-\left\{\Lambda\left(x+\frac{\alpha}{2}\right)+\Lambda\left(x-\frac{\alpha}{2}\right)\right\}\label{impotE}
\end{align}
since $\Lambda(s)=\frac{1}{2}\mathrm{Im}(\mathrm{Li}_2(e^{2\sqrt{-1}s}))$ holds. 

We also consider the summands of the colored Jones invariant up to a factor of $\frac{1}{\{1\}}$. The colored Jones invariant $V_j^{(r)}(E)$ has the summand
\[A_k(E;j)=\frac{1}{\{1\}}\frac{\{2j+1+k\}!}{\{2j-k\}!}.\]
Let $R_k(E;j)$ be the ratio $A_k(E;j)/A_{k-1}(E;j)$ for $k\geq 1$, where $R_{0}(E;j)=A_0(E;j)$. We know
\[A_k(E;j)=\prod_{\nu=0}^kR_\nu(E;j)\]
holds and then $A_k(E;j)$ forms a real-valued sequence indexed by $k$. By focusing on the ratio $R_k(E;j)$, the sign of $A_k(E;j)$ is alternating whenever $R_k(E;j)<0$, and remains constant whenever $R_k(E;j)>0$. We compute the ratio as follows: 
\begin{align*}
	R_k(E;j)&=\{2j+1+k\}\{2j+1-k\}\\
	&=-4\sin\frac{2\pi(2j+1+k)}{r}\sin\frac{2\pi(2j+1-k)}{r}\\
	&=2\left\{\cos\frac{4\pi(2j+1)}{r}-\cos\frac{4\pi k}{r}\right\}.
\end{align*}
The ratio $R_k(E;j)$ is never equal to 0 since $0\leq k\leq \min\{2j,r-(2j+1)-1\}$. 

Now we consider the range of the summation of the colored Jones invariant for $E$. Define sets
\[I_1=\left\{0,1,\ldots,\left\lfloor\left|\frac{r}{2}-(2j+1)\right|\right\rfloor\right\},\quad I_2=\left\{\left\lfloor\left|\frac{r}{2}-(2j+1)\right|\right\rfloor+1,\ldots,k_{\max}\right\},\]
and $I=I_1\cup I_2$. Note that these sets depend on $r$.
\begin{lemma}\label{sgnchgE}
The sequence of $A_k(E;j)$ is alternating in $I_1$ and has constant sign in $I_2$. 
\end{lemma}
\begin{proof}
This follows immediately from the sign changes in the trigonometric factors of $R_k(E;j)$.
\end{proof}
We can separate the sum into the alternating part and the constant sign part:
\[V_{j}^{(r)}(E)=\sum_{k\in I_1}A_k(E;j)+\sum_{k\in I_2}A_k(E;j).\]

Taking $r\to\infty$ yields $2j\sim\frac{r}{2}\pm\frac{\alpha}{4\pi}r$ and the condition $\frac{r}{2}\leq 2j+1$ yields $k_{\max}=r-(2j+1)-1$. Then $\frac{2\pi}{r}k_{\max}$ asymptotically tends to $\pi-\frac{\alpha}{2}$ when $r\to\infty$. Moreover, it also follows from the condition $\frac{r}{2}>2j+1$ that $\frac{2\pi(2j)}{r}\sim\pi-\frac{\alpha}{2}$ $(r\to\infty)$. 

\begin{lemma}
Assume that $r$ is sufficiently large.
\begin{enumerate}[(i)]
\item If $0\leq\alpha\leq\frac{\pi}{3}$, then there exist integers $0\leq k_1\leq k_2\leq k_{\max}$ such that $|A_k(E;j)|$ is non-increasing for $0\leq k\leq k_1$, non-decreasing for $k_1\leq k\leq k_2$, and non-increasing for $k_2\leq k\leq k_{\max}$. In particular, $|A_k(E;j)|$ attains its minimal value at $k=k_1$ and its maximal value at $k=k_2$.
\item If $\frac{\pi}{3}<\alpha<\frac{2\pi}{3}$, then there exist integers $0\leq k_1'\leq k_2'\leq k_3'\leq k_{\max}$ such that $|A_k(E;j)|$ is non-decreasing on $[0,k_1']$, non-increasing on $[k_1',k_2']$, non-decreasing on $[k_2',k_3']$, and non-increasing on $[k_3',k_{\max}]$. In particular, $|A_k(E;j)|$ has two maximal values at $k=k_1'$ and $k=k_3'$ separated by the minimal value at $k=k_2'$.
\end{enumerate}
\end{lemma}
\begin{proof}
As $r\to\infty$, the solutions of $|R_k(E;j)|=1$ converge to the critical points of $2\,\mathrm{Im}(\Phi_E(x))$ listed below. Since it is known that
\begin{equation}
\log|\{n\}!|=-\frac{r}{2\pi}\Lambda\left(\frac{2\pi n}{r}\right)+O(\log r)\quad(r\to\infty),\label{logfact}
\end{equation}
we get
\begin{equation}
	\log|A_k(E;j)|=\frac{r}{2\pi}\left[-\left\{\Lambda\left(\frac{2\pi(k+2j+1)}{r}\right)+\Lambda\left(\frac{2\pi(k-2j)}{r}\right)\right\}\right]+O(\log r)\quad(r\to\infty),\notag\label{sumasymp}
\end{equation}
and so we have 
\begin{equation}
\frac{4\pi}{r}\log|A_k(E;j)|\sim2\,\mathrm{Im}(\Phi_E(x))\quad(r\to\infty).\label{coincidenceE}
\end{equation}
By considering the first and the second derivatives of $2\,\mathrm{Im}(\Phi_E(x))$, we have the following: if $0\leq\alpha\leq\frac{\pi}{3}$, then the function has the minimal value at $x=\frac{1}{2}\arccos\left(\cos\alpha-\frac{1}{2}\right)$ and the maximal value at $x=\pi-\frac{1}{2}\arccos\left(\cos\alpha-\frac{1}{2}\right)$; if $\frac{\pi}{3}<\alpha<\frac{2\pi}{3}$, then the function has the minimal value at $x=\frac{1}{2}\arccos\left(\cos\alpha-\frac{1}{2}\right)$ and the maximal values at $x=\frac{1}{2}\arccos\left(\cos\alpha+\frac{1}{2}\right)$ and $x=\pi-\frac{1}{2}\arccos\left(\cos\alpha-\frac{1}{2}\right)$. In particular, there exist two points around the extremum point which corresponds to $\pi-\frac{1}{2}\arccos\left(\cos\alpha-\frac{1}{2}\right)$ in the interval $\left[\frac{\alpha}{2},\pi-\frac{\alpha}{2}\right]$. For a maximal (resp. minimal) value of $2\,\mathrm{Im}(\Phi_E(x))$, we know that there exist two points of $|A_k(E;j)|$ around the point corresponding to the extremum point; we obtain the conclusion by choosing the larger (resp. smaller) one of the two values at the two points.
\end{proof}

\subsection{Proofs of Theorem \ref{MainTheorem1}}\label{PMT1}
In this section, we give a proof of the first main result as follows:

\mainthmF*
\noindent Note that the limit formula of Theorem \ref{MainTheorem1} from Conjecture \ref{mainconjecture} is
\[\lim_{\substack{r\to\infty\\r:\mathrm{odd}}}\frac{4\pi}{r}\log|V_{j}^{(r)}(E)|=\mathrm{Vol}(M_{\alpha}(E)).\]

\begin{lemma}\label{constantsign}
It holds that
\begin{equation}
\lim_{\substack{r\to\infty\\r:\mathrm{odd}}}\frac{4\pi}{r}\log\left|\sum_{k\in I_2}A_k(E;j)\right|=\lim_{\substack{r\to\infty\\r:\mathrm{odd}}}\frac{4\pi}{r}\log\max_{k\in I_2}|A_k(E;j)|.\label{cseq}
\end{equation}
\end{lemma}
\begin{proof}
We know
\[\left|\sum_{k\in I_2}A_k(E;j)\right|\leq\sum_{k\in I_2}|A_k(E;j)|\leq\#I_2\max_{k\in I_2}|A_k(E;j)|.\]
Since $A_k(E;j)$ has a constant sign in $I_2$, we have
\begin{equation}
\left|\sum_{k\in I_2}A_k(E;j)\right|=\sum_{k\in I_2}|A_k(E;j)|\geq\max_{k\in I_2}|A_k(E;j)|.\notag\label{secsum}
\end{equation}
Hence, we have
\begin{equation}
\max_{k\in I_2}|A_k(E;j)|\leq\left|\sum_{k\in I_2}A_k(E;j)\right|\leq\#I_2\max_{k\in I_2}|A_k(E;j)|.\label{SqueezeE}
\end{equation}
Thus, from \eqref{SqueezeE}, since $\#I_2=O(r)$ $(r\to\infty)$, we obtain \eqref{cseq}.
\end{proof}

Next, we briefly prepare the geometric concepts required to see that the leading term corresponds to the hyperbolic volume. A 3-dimensional cone-manifold is a Riemannian 3-dimensional manifold of constant sectional curvature with cone-type singular set along simple closed geodesics, and it is modeled in hyperbolic, spherical, or Euclidean structure depending on the curvature. For the conjecture, we consider the hyperbolic volumes of 3-dimensional hyperbolic cone-manifolds along hyperbolic links. 
Note that it is known that they are given by the imaginary parts of the potential functions of the colored Jones invariants for hyperbolic links evaluated at the suitable saddle point (see \cite{Sawabe2023}).

\begin{proposition}[{\cite[Theorem 6.3 (ii)]{Abrosimov-Mednykh2021}}]\label{VolEformula}
The hyperbolic volume of $M_\alpha(E)$ is given by the formula 
\begin{equation}
\mathrm{Vol}(M_\alpha(E))=2\left\{\Lambda\left(\theta_E+\frac{\alpha}{2}\right)+\Lambda\left(\theta_E-\frac{\alpha}{2}\right)\right\},\label{VolE}
\end{equation}
where $\theta_E=\frac{1}{2}\arccos\left(\cos\alpha-\frac{1}{2}\right)$. 
\end{proposition}

\begin{remark}
The cone manifold $M_\alpha(E)$ is hyperbolic when $0\leq\alpha<\frac{2\pi}{3}$.
\end{remark}

From \eqref{coincidenceE}, \eqref{cseq}, and \eqref{VolE}, we know
\begin{equation}
\lim_{\substack{r\to\infty\\r:\mathrm{odd}}}\frac{4\pi}{r}\log\left|\sum_{k\in I_2}A_k(E;j)\right|=\mathrm{Vol}(M_\alpha(E)).\label{ItwoVolE}
\end{equation}

We mainly discuss the case where $0<\alpha<\alpha_0$. The case $\alpha=0$ follows similarly by degeneration (see also \cite{DKY2018}). To prove Theorem \ref{MainTheorem1}, we consider the evaluation of the absolute value of the alternating summation $\left|\sum_{k\in I_1}A_k(E;j)\right|$. Let $x_0$ be the point corresponding to the hyperbolic volume, namely, we put $x_0=\pi-\frac{1}{2}\arccos\left(\cos\alpha-\frac{1}{2}\right)$. We suppose that $\mathrm{Im}\left(\Phi_E\left(\frac{\alpha}{2}\right)\right)<\mathrm{Im}(\Phi_E(x_0))=\frac{1}{2}\mathrm{Vol}(M_\alpha(E))$ holds. Choose $U_E$ so that $\max\left\{\mathrm{Im}\left(\Phi_E\left(\frac{\alpha}{2}\right)\right),0\right\}<U_E<\mathrm{Im}(\Phi_E(x_0))$. We take the interval $I(\alpha)=\left[0,\frac{\alpha}{2}\right]$ and its subintervals $I'(\alpha)=\left[\varepsilon,\frac{\alpha}{2}-\varepsilon\right]$ and $I''(\alpha)=\left[2\varepsilon,\frac{\alpha}{2}-2\varepsilon\right]$ for a small $\varepsilon>0$ such that $\mathrm{Im}(\Phi_E(x))<U_E$ for every $x\in I(\alpha)\backslash I''(\alpha)$. Let $g_\alpha$ be a smooth function on $\mathbb{R}$ such that $g_\alpha(t)=0$ if $t$ is outside $I'(\alpha)$ and $g_\alpha(t)=1$ if $t$ is in $I''(\alpha)$. 

Now we define a holomorphic function $\varphi_r(z)$ on $\left\{z\in\mathbb{C}\,\middle|\,-\frac{\pi}{r}<\mathrm{Re}(z)<\pi+\frac{\pi}{r}\right\}$ by
\[\varphi_r(z)=\int_{-\infty}^\infty\frac{e^{(2z-\pi)x}}{4x\sinh(\pi x)\sinh (2\pi x/r)}dx,\]
where the above integrand has poles on the imaginary axis 
and we choose the path of the integral
\[(-\infty,-c]\cup\{z\in\mathbb{C}\,|\,|z|=c,\,\mathrm{Im}(z)\geq0\}\cup[c,\infty)\]
for some $c\in(0,1)$ to avoid the pole at $0$. Note that $\varphi_r(z)$ is called the quantum dilogarithm function. It is known that the formula 
\[(s^2)_n=\left\{\begin{array}{l l}\displaystyle\exp\left[\varphi_r\left(\frac{\pi}{r}\right)-\varphi_r\left(\frac{2\pi n}{r}+\frac{\pi}{r}\right)\right]&\displaystyle\left(0\leq n\leq\frac{r-1}{2}\right),\vspace{0.25em}\\\displaystyle\exp\left[\varphi_r\left(\frac{\pi}{r}\right)-\varphi_r\left(\frac{2\pi n}{r}+\frac{\pi}{r}-\pi\right)+\log 2\right]&\displaystyle\left(\frac{r-1}{2}<n<\frac{2r-1}{2}\right)\end{array}\right.\]
holds (see \cite{Chen-Murakami2023}). We rewrite the summation as
\begin{align*}
\sum_{k\in I_1}A_k(E;j)
&=\sum_{k\in I_1}\frac{1}{\{1\}}\exp\left[\frac{r}{2\pi\sqrt{-1}}\left\{-2\pi\frac{2\pi k}{r}+4\frac{2\pi j}{r}\frac{2\pi k}{r}+2\frac{2\pi}{r}\frac{2\pi j}{r}+2\frac{2\pi}{r}\frac{2\pi k}{r}+\frac{2\pi^2}{r}\right.\right.\\
&\left.\left.\hspace{12em}-\frac{2\pi\sqrt{-1}}{r}\varphi_r\left(\frac{2\pi(2j+1+k)}{r}+\frac{\pi}{r}-b\pi\right)\right.\right.\\
&\left.\left.\hspace{15em}+\frac{2\pi\sqrt{-1}}{r}\varphi_r\left(\frac{2\pi(2j-k)}{r}+\frac{\pi}{r}-b\pi\right)\right\}\right],
\end{align*}
where $b$ is either 0 or 1, depending on the range of the indices of the Pochhammer symbol. It is also known that $\frac{2\pi\sqrt{-1}}{r}\varphi_r(z)$ uniformly converges to $\frac{1}{2}\mathrm{Li}_2(e^{2\sqrt{-1}z})$ in the domain $\{z\in\mathbb{C}\,|\,d\leq\mathrm{Re}(z)\leq \pi-d,|\mathrm{Im}(z)|\leq R\}$ for any sufficiently small $d>0$ and any $R>0$. We put 
\begin{align*}
\Phi_E^{r}\left(\frac{2\pi k}{r}\right)&=-2\pi\frac{2\pi k}{r}+4\frac{2\pi j}{r}\frac{2\pi k}{r}+2\frac{2\pi}{r}\frac{2\pi j}{r}+2\frac{2\pi}{r}\frac{2\pi k}{r}+\frac{2\pi^2}{r}\\
&\hspace{2em}-\frac{2\pi\sqrt{-1}}{r}\varphi_r\left(\frac{2\pi(2j+1+k)}{r}+\frac{\pi}{r}-b\pi\right)+\frac{2\pi\sqrt{-1}}{r}\varphi_r\left(\frac{2\pi(2j-k)}{r}+\frac{\pi}{r}-b\pi\right).
\end{align*}
Moreover, for a sufficiently large $r$, we regard the sequence of the functions as 
\[\Phi_E^{r}(z)=-\frac{2\pi\sqrt{-1}}{r}\varphi_r\left((1-b)\pi+\frac{\alpha}{2}+z+\frac{3\pi}{r}\right)+\frac{2\pi\sqrt{-1}}{r}\varphi_r\left((1-b)\pi+\frac{\alpha}{2}-z+\frac{\pi}{r}\right)+\alpha z.\]
We define the function $h_{\alpha,r}(x)$ to be the product $g_\alpha\left(\frac{2\pi x}{r}\right)\exp\left[\frac{r}{2\pi\sqrt{-1}}\Phi_E^r\left(\frac{2\pi x}{r}\right)\right]$, and we consider Fourier coefficients
\[\widehat{h_{\alpha,r}}(m)=\int_\mathbb{R}h_{\alpha,r}(x)e^{-2\pi mx\sqrt{-1}}dx\]
to use the Poisson summation formula. Recall that a rapidly decreasing function $h(x)$ satisfies the Poisson summation formula
\[\sum_{m\in\mathbb{Z}}h(m)=\sum_{m\in\mathbb{Z}}\widehat{h}(m).\]

\begin{lemma}\label{PoiE}
The function $h_{\alpha,r}(x)$ is a rapidly decreasing function. Therefore, it holds that
\[\sum_{m\in\mathbb{Z}}h_{\alpha,r}(m)=\sum_{m\in\mathbb{Z}}\widehat{h_{\alpha,r}}(m).\]
\end{lemma}

Here we prove the following lemma.

\begin{lemma}\label{evaE}
There exists a constant $M>0$ such that
\[\left|\sum_{m\in\mathbb{Z}}\widehat{h_{\alpha,r}}(m)\right|\leq Mre^{\frac{r}{2\pi}U_E}\]
for a sufficiently large $r$.
\end{lemma}
\begin{proof}
Assume that $m\neq0$. Then we have
\begin{align*}
\widehat{h_{\alpha,r}}(m)&=\int_\mathbb{R}g_\alpha\left(\frac{2\pi x}{r}\right)\exp\left[\frac{r}{2\pi\sqrt{-1}}\Phi_E^r\left(\frac{2\pi x}{r}\right)\right]e^{-2\pi mx\sqrt{-1}}dx\\
&=\frac{r}{2\pi}\int_\mathbb{R}g_\alpha(x)\exp\left[\frac{r}{2\pi\sqrt{-1}}\Phi_E^r(x)\right]e^{-mrx\sqrt{-1}}dx\\
&=\frac{1}{2\pi m\sqrt{-1}}\int_\mathbb{R}\left\{g_\alpha'(x)+\frac{r}{2\pi\sqrt{-1}}g_\alpha(x){\Phi_E^{r}}'(x)\right\}\exp\left[\frac{r}{2\pi\sqrt{-1}}\Phi_E^r(x)\right]e^{-mrx\sqrt{-1}}dx\\
&=\frac{1}{2\pi m^2}\int_\mathbb{R}\left[\frac{1}{r}g_\alpha''(x)+\frac{1}{2\pi\sqrt{-1}}\left\{2g_\alpha'(x){\Phi_E^r}'(x)+g_\alpha(x){\Phi_E^r}''(x)+\frac{r}{2\pi\sqrt{-1}}g_\alpha(x){\Phi_E^r}'(x)^2\right\}\right]\\
&\hspace{23em}\times\exp\left[\frac{r}{2\pi\sqrt{-1}}\{\Phi_E^r(x)+2\pi mx\}\right]dx.
\end{align*}
Since $g_\alpha$ is supported in $I'(\alpha)$, we split the last integral into the core part over $I''(\alpha)$ and the transition part over $I'(\alpha)\backslash I''(\alpha)$. On $I''(\alpha)$, we have $g_\alpha(x)=1$ and $g_\alpha'(x)=g_\alpha''(x)=0$. Hence, the integrand of the core part is holomorphic in a neighborhood of the contours considered below. We deform only this core integral, while the transition part is kept on the real axis. Thus, there exists a constant $M'>0$ which does not depend on $r$ such that
\begin{align}
|\widehat{h_{\alpha,r}}(m)|\leq&\,\frac{M'r}{2\pi m^2}\int_{I(\alpha)}\exp\left[\frac{r}{2\pi}\mathrm{Im}(\Phi_E^r(x)+2\pi mx)\right]dx.\label{nonzeroevE}
\end{align}

Recall that $\Phi_E^r(z)$ uniformly converges to $\Phi_E(z)$ in a suitable domain. Hence, assume that $r$ is sufficiently large. Henceforth, we use a suitable contour of $\mathrm{Im}(\Phi_E(x)+2\pi mx)$ to deform the path of integration for any integer $m$. Let $C_{-1}(\alpha)$ denote the path obtained from $I(\alpha)$ by keeping $I'(\alpha)\backslash I''(\alpha)$ on the real axis and deforming only $I''(\alpha)$ into the first quadrant so that $\mathrm{Im}(\Phi_E(x) - 2\pi x)<U_E$ on the deformed part. Similarly, let $C_0(\alpha)$ be the path obtained from $I(\alpha)$ by keeping $I'(\alpha)\backslash I''(\alpha)$ on the real axis and deforming only $I''(\alpha)$ into the fourth quadrant so that $\mathrm{Im}(\Phi_E(x))<U_E$ on the deformed part. The paths $C_{-1}(\alpha)$ and $C_0(\alpha)$ are shown in Figure \ref{FigPathE}.
\begin{lemma}\label{pathsE}
The paths $C_{-1}(\alpha)$ and $C_0(\alpha)$ exist. Moreover, $C_{-1}(\alpha)$ lies in the first quadrant or on the real axis and $C_0(\alpha)$ lies in the fourth quadrant or on the real axis.
\end{lemma}
A proof of Lemma \ref{pathsE} is given in Section \ref{lemmaproofs}.
\begin{figure}[htbp]
\centering\includegraphics[width=7.5cm]{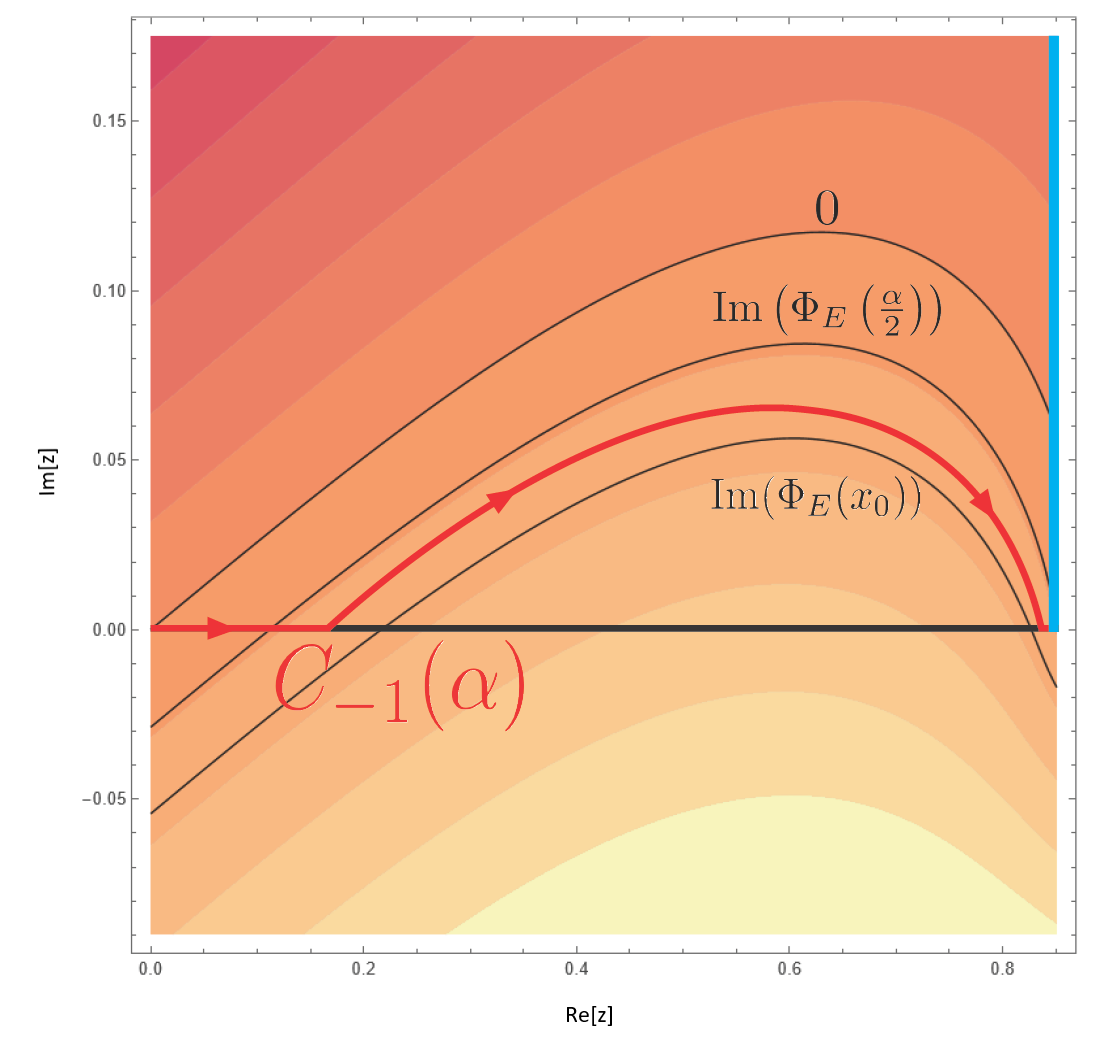}
\centering\includegraphics[width=7.25cm]{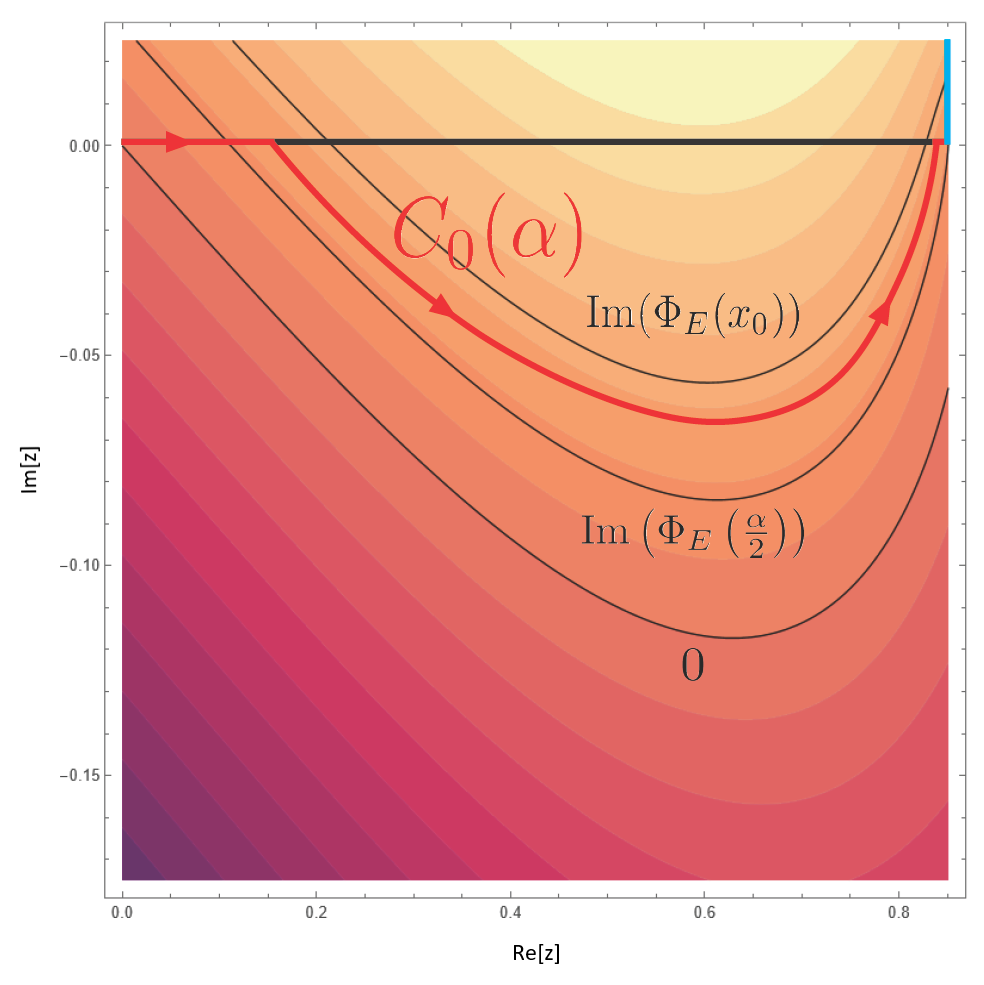}
\caption{The paths of integration along the contours and the real axis: $C_{-1}(\alpha)$ (left) and $C_0(\alpha)$ (right) at $\alpha=\frac{13\pi}{24}$. The red oriented lines are the integration paths and the blue lines are the branch cuts. The black curves indicate the level sets of $\mathrm{Im}(\Phi_E(x)-2\pi x)$ and $\mathrm{Im}(\Phi_E(x))$. }\label{FigPathE}
\end{figure}

From \eqref{nonzeroevE} and Lemma \ref{pathsE}, we get
\[|\widehat{h_{\alpha,r}}(-1)|\leq\frac{M'r}{2\pi}\int_{C_{-1}(\alpha)}\exp\left[\frac{r}{2\pi}\mathrm{Im}(\Phi_E^r(x)-2\pi x)\right]dx\leq\frac{M'r}{2\pi}\ell(C_{-1}(\alpha))e^{\frac{r}{2\pi}U_E}.\]
Furthermore, since $\mathrm{Im}(x)\geq0$ for $x$ in $C_{-1}(\alpha)$, $\mathrm{Im}(\Phi_E(x)+2\pi mx)\leq\mathrm{Im}(\Phi_E(x)-2\pi x)$ for $m<-1$, we obtain
\begin{equation}
|\widehat{h_{\alpha,r}}(m)|\leq\frac{M'r}{2\pi m^2}\ell(C_{-1}(\alpha))e^{\frac{r}{2\pi}U_E}\quad(m<-1).\label{nonzeroE}
\end{equation}
We next consider the case where $m=0$:
\begin{align*}
\widehat{h_{\alpha,r}}(0)&=\int_\mathbb{R}g_\alpha\left(\frac{2\pi x}{r}\right)\exp\left[\frac{r}{2\pi\sqrt{-1}}\Phi_E^r\left(\frac{2\pi x}{r}\right)\right]dx\\
&=\frac{r}{2\pi}\left\{\int_{I''(\alpha)}\exp\left[\frac{r}{2\pi\sqrt{-1}}\Phi_E^r(x)\right]dx+\int_{I'(\alpha)\backslash I''(\alpha)}g_\alpha(x)\exp\left[\frac{r}{2\pi\sqrt{-1}}\Phi_E^r(x)\right]dx\right\}.
\end{align*}
The integrand of the first integral is holomorphic, and hence we may deform only this integral to the deformed part of $C_0(\alpha)$. The second integral is kept on the real axis. Therefore, we know
\begin{equation}
|\widehat{h_{\alpha,r}}(0)|\leq\frac{r}{2\pi}\{\ell(C_0(\alpha))+\ell(I'(\alpha)\backslash I''(\alpha))\}e^{\frac{r}{2\pi}U_E}.\notag
\end{equation}
Applying the same decomposition for $m>0$, and since $\mathrm{Im}(x)\leq 0$ for $x$ in $C_0(\alpha)$, $\mathrm{Im}(\Phi_E(x)+2\pi mx)\leq\mathrm{Im}(\Phi_E(x))$ for $m>0$, and then we obtain
\begin{equation}
|\widehat{h_{\alpha,r}}(m)|\leq\frac{M''r}{2\pi m^2}\ell(C_0(\alpha))e^{\frac{r}{2\pi}U_E}\quad(m>0)\label{zeroE}
\end{equation}
for some $M''>0$. 

From \eqref{nonzeroE} and \eqref{zeroE}, there exists $M>0$ such that
\[\left|\sum_{m\in\mathbb{Z}}\widehat{h_{\alpha,r}}(m)\right|\leq Mre^{\frac{r}{2\pi}U_E}.\]
\end{proof}

We are now in a position to complete the proof of Theorem \ref{MainTheorem1}.
\begin{proof}[Proof of Theorem \ref{MainTheorem1}]
First, 
\[\left|\sum_{k\in I_1}A_k(E;j)\right|=\frac{1}{|\{1\}|}\left|\sum_{k\in I_1}\exp\left[\frac{r}{2\pi\sqrt{-1}}\Phi_E^r\left(\frac{2\pi k}{r}\right)\right]\right|\]
and $\lim_{\substack{r\to\infty\\r:\mathrm{odd}}}\frac{4\pi}{r}\log\frac{1}{|\{1\}|}=0.$ For a sufficiently large $r$, we have
\begin{align*}
&\left|\sum_{k\in I_1}\exp\left[\frac{r}{2\pi\sqrt{-1}}\Phi_E^r\left(\frac{2\pi k}{r}\right)\right]\right|\\
&=\left|\sum_{m\in\mathbb{Z}}h_{\alpha,r}(m)+\sum_{k\in \left.I_1\middle\backslash\left(\mathbb{Z}\cap \frac{r}{2\pi}I''(\alpha)\right)\right.}(1-g_\alpha)\left(\frac{2\pi k}{r}\right)\exp\left[\frac{r}{2\pi\sqrt{-1}}\Phi_E^r\left(\frac{2\pi k}{r}\right)\right]\right|\\
&\leq\left|\sum_{m\in\mathbb{Z}}h_{\alpha,r}(m)\right|+\sum_{k\in \left.I_1\middle\backslash\left(\mathbb{Z}\cap \frac{r}{2\pi}I''(\alpha)\right)\right.}\exp\left[\frac{r}{2\pi}\mathrm{Im}\left(\Phi_E^r\left(\frac{2\pi k}{r}\right)\right)\right]\\
&\hspace{1em}\leq\left|\sum_{m\in\mathbb{Z}}h_{\alpha,r}(m)\right|+\#\left.I_1\middle\backslash\left(\mathbb{Z}\cap \frac{r}{2\pi}I''(\alpha)\right)\right.e^{\frac{r}{2\pi}U_E}.
\end{align*}
Furthermore, we also have
\[\left|\sum_{m\in\mathbb{Z}}h_{\alpha,r}(m)\right|=\left|\sum_{m\in\mathbb{Z}}\widehat{h_{\alpha,r}}(m)\right|\leq Mre^{\frac{r}{2\pi}U_E}\]
from Lemma \ref{PoiE} and Lemma \ref{evaE}. Thus, we know
\begin{align}
\frac{\left|\sum_{k\in I_1}A_k(E;j)\right|}{\left|\sum_{k\in I_2}A_k(E;j)\right|}\leq Mr^2e^{\frac{r}{2\pi}(U_E-\mathrm{Im}(\Phi_E(x_0)))}\to0\quad(r\to\infty,r:\mathrm{odd})\label{expsmallE}
\end{align}
since we assume that $U_E<\mathrm{Im}(\Phi_E(x_0))$ and Lemma \ref{constantsign} holds. Now it clearly holds that
\begin{equation}
\left|\sum_{k\in I_2}A_k(E;j)\right|-\left|\sum_{k\in I_1}A_k(E;j)\right|\leq|V_{j}^{(r)}(E)|\leq\left|\sum_{k\in I_1}A_k(E;j)\right|+\left|\sum_{k\in I_2}A_k(E;j)\right|,\label{triangleE}
\end{equation}
and we have
\begin{align}
&\lim_{\substack{r\to\infty\\r:\mathrm{odd}}}\frac{4\pi}{r}\log\left|\left|\sum_{k\in I_2}A_k(E;j)\right|\pm\left|\sum_{k\in I_1}A_k(E;j)\right|\right|\notag\\
&=\lim_{\substack{r\to\infty\\r:\mathrm{odd}}}\frac{4\pi}{r}\left\{\log\left|\sum_{k\in I_2}A_k(E;j)\right|+\log\left|1\pm\frac{\left|\sum_{k\in I_1}A_k(E;j)\right|}{\left|\sum_{k\in I_2}A_k(E;j)\right|}\right|\right\}\notag\\
&=\mathrm{Vol}(M_\alpha(E))+\lim_{\substack{r\to\infty\\r:\mathrm{odd}}}\frac{4\pi}{r}\log\left|1\pm\frac{\left|\sum_{k\in I_1}A_k(E;j)\right|}{\left|\sum_{k\in I_2}A_k(E;j)\right|}\right|.\label{lastE}
\end{align}
From \eqref{expsmallE}, \eqref{triangleE}, and \eqref{lastE}, we finally obtain
\[\lim_{\substack{r\to\infty\\r:\mathrm{odd}}}\frac{4\pi}{r}\log|V_{j}^{(r)}(E)|=\mathrm{Vol}(M_\alpha(E)).\]
Note that the condition $\mathrm{Im}\left(\Phi_E\left(\frac{\alpha}{2}\right)\right)<\mathrm{Im}(\Phi_E(x_0))$ is numerically equivalent to $\alpha<1.7647826174\ldots$ by a computation.
\end{proof}

\section{The case of the Borromean rings}\label{SecB}
Although the colored Jones invariant of the Borromean rings is more complicated than that of the figure-eight knot, the same general strategy still applies.

\subsection{The colored Jones invariant for $B$}
The colored Jones invariant for the Borromean rings $B$ is
\begin{align*}
	V_{j_1,j_2,j_3}^{(r)}(B)&=\sum_{k=0}^{\min_{i=1,2,3}\{2j_i,r-(2j_i+1)-1\}}(-1)^k\frac{\{2j_1+1+k\}!\{2j_2+1+k\}!\{2j_3+1+k\}!}{\{1\}\{2j_1-k\}!\{2j_2-k\}!\{2j_3-k\}!}\left(\frac{\{k\}!}{\{2k+1\}!}\right)^2
\end{align*}
as in \cite{Habiro2008,Murakami-Tran2025} (see also \cite{Habiro2000}). The above formula has the following form:
\begin{align*}
V_{j_1,j_2,j_3}^{(r)}(B)&=\sum_{k=0}^{k_{\max}}\frac{1}{\{1\}}(-1)^{k+1}s^{-(2j_1+1)(2k+1)-(2j_2+1)(2k+1)-(2j_3+1)(2k+1)+(3k+2)(k+1)}\\
	&\hspace{24em}\times\frac{(s^2)_{k}^2}{(s^2)_{2k+1}^2}\prod_{i=1}^3\frac{(s^2)_{2j_i+1+k}}{(s^2)_{2j_i-k}},
\end{align*}
where $k_{\max}=\min_{i=1,2,3}\{2j_i,r-(2j_i+1)-1\}$. The possible values of $n$ in the form $(s^2)_n$ above for the case of $B$ are $k$, $2k+1$, $2j_i+1+k$, and $2j_i-k$ $(i=1,2,3)$. Then we have
\begin{align*}
	V_{j_1,j_2,j_3}^{(r)}(B)&\sim\sum_{k=0}^{k_{\max}}\frac{1}{\{1\}}\exp\left[\frac{r}{2\pi\sqrt{-1}}\left\{-\pi\frac{2\pi k}{r}+4\sum_{i=1}^3\frac{2\pi j_i}{r}\frac{2\pi k}{r}-3\left(\frac{2\pi k}{r}\right)^2\right.\right.\\
	&\hspace{12em}-\frac{1}{2}\sum_{i=1}^3(\mathrm{Li}_2(e^{2\sqrt{-1}\frac{2\pi(2j_i+1+k)}{r}})-\mathrm{Li}_2(e^{2\sqrt{-1}\frac{2\pi(2j_i-k)}{r}}))\\
	&\hspace{14em}\left.\left.-\mathrm{Li}_2(e^{2\sqrt{-1}\frac{2\pi k}{r}})+\mathrm{Li}_2(e^{2\sqrt{-1}\frac{2\pi(2k+1)}{r}})\right\}\right]\quad(r\to\infty).
\end{align*}
We regard $\frac{2\pi k}{r}$ as the continuous real variable $x$ for a large $r$. 

By extending $x$ to a complex variable $z$, the exponents with $\frac{r}{2\pi\sqrt{-1}}$ of the exponential functions in the calculated summands turn into the potential function
\begin{align*}
\Phi_B(\alpha_1,\alpha_2,\alpha_3;z)=&\sum_{i=1}^3\left\{-\frac{1}{2}\mathrm{Li}_2(e^{2\sqrt{-1}(\alpha_i/2+z)})+\frac{1}{2}\mathrm{Li}_2(e^{2\sqrt{-1}(\alpha_i/2-z)})\right\}\\
&\hspace{7em}-\mathrm{Li}_2(e^{2\sqrt{-1}z})+\mathrm{Li}_2(e^{4\sqrt{-1}z})+(\alpha_1+\alpha_2+\alpha_3+5\pi)z-3z^2.
\end{align*}
As in the case of the figure-eight knot, the symbol $\Phi_B(z)$ denotes this function. Note that there is also ambiguity as we see in Section \ref{tcjiaipffE}. 

\begin{remark}
The potential function $\Phi_B(z)$ has the branch cuts along
\begin{align*}
&\left\{\pm\pi n\mp\frac{\alpha_i}{2}\mp\sqrt{-1}R\,\middle|\,R\geq0\right\},\quad\left\{\pi n-\sqrt{-1}R\,\middle|\,R\geq0\right\},\quad\left\{\frac{\pi n}{2}-\sqrt{-1}R\,\middle|\,R\geq0\right\}
\end{align*}
in the complex plane ($i=1,2,3;n\in\mathbb{Z}$).
\end{remark}

If we consider the potential functions $\Phi_B(z)$ for $z=x\in\mathbb{R}$, we get 
\begin{align*}
	\mathrm{Im}(\Phi_B(x))&=-\sum_{i=1}^3\left\{\Lambda\left(\frac{\alpha_i}{2}+x\right)-\Lambda\left(\frac{\alpha_i}{2}-x\right)\right\}-2\Lambda(x)+2\Lambda(2x)\\
	&=-\sum_{i=1}^3\left\{\Lambda\left(\frac{\alpha_i}{2}+x\right)-\Lambda\left(\frac{\alpha_i}{2}-x\right)\right\}+2\Lambda\left(\frac{\pi}{2}+x\right)+2\Lambda\left(\frac{\pi}{2}-x\right)+\Lambda(x)-\Lambda(-x)\\
	&=-\Delta\left(\frac{\alpha_1}{2},x\right)-\Delta\left(\frac{\alpha_2}{2},x\right)-\Delta\left(\frac{\alpha_3}{2},x\right)+2\Delta\left(\frac{\pi}{2},x\right)+\Delta(0,x)
\end{align*}
since $\Lambda(s)=\frac{1}{2}\mathrm{Im}(\mathrm{Li}_2(e^{2\sqrt{-1}s}))$ and $\Lambda(s)=\frac{1}{2}\Lambda(2s)-\Lambda\left(s+\frac{\pi}{2}\right)$ hold. 

The colored Jones invariant $V_{j_1,j_2,j_3}^{(r)}(B)$ has the summand
\[A_k(B;j_1,j_2,j_3)=(-1)^k\frac{\{2j_1+1+k\}!\{2j_2+1+k\}!\{2j_3+1+k\}!}{\{1\}\{2j_1-k\}!\{2j_2-k\}!\{2j_3-k\}!}\left(\frac{\{k\}!}{\{2k+1\}!}\right)^2.\]
Let $R_k(B;j_1,j_2,j_3)$ be the ratio $A_k(B;j_1,j_2,j_3)/A_{k-1}(B;j_1,j_2,j_3)$ for $k\geq 1$, where $R_{0}(B;j_1,j_2,j_3)=A_0(B;j_1,j_2,j_3)$. We know
\[A_k(B;j_1,j_2,j_3)=\prod_{\nu=0}^kR_\nu(B;j_1,j_2,j_3)\]
holds and then $A_k(B;j_1,j_2,j_3)$ forms a real-valued sequence indexed by $k$. The sign of $A_k(B;j_1,j_2,j_3)$ is alternating whenever $R_k(B;j_1,j_2,j_3)<0$, and remains constant whenever $R_k(B;j_1,j_2,j_3)>0$. Now we compute the ratio as follows:
\begin{align*}
	R_k(B;j_1,j_2,j_3)&=-16\frac{\sin^2\frac{2\pi k}{r}}{\sin^2\frac{2\pi(2k+1)}{r}\sin^2\frac{4\pi k}{r}}\prod_{i=1}^3\sin\frac{2\pi(2j_i+1+k)}{r}\sin\frac{2\pi(2j_i+1-k)}{r}\\
	&=2\frac{\sin^2\frac{2\pi k}{r}}{\sin^2\frac{2\pi(2k+1)}{r}\sin^2\frac{4\pi k}{r}}\prod_{i=1}^3\left\{\cos\frac{4\pi(2j_i+1)}{r}-\cos\frac{4\pi k}{r}\right\}.
\end{align*}
The ratio $R_k(B;j_1,j_2,j_3)$ is never equal to 0 since $0\leq k\leq \min_{i=1,2,3}\{2j_i,r-(2j_i+1)-1\}$. 

We may assume that $j_1\leq j_2\leq j_3$ without loss of generality by symmetry. A partition of the range $I=\{0,1,\ldots,k_{\max}\}$ has four subsets $I_1$, $I_2$, $I_3$, and $I_4$ depending on the sign changes of the summand. For example, if $\frac{r}{2}>2j_i+1$ for all $i$, then 
\begin{align*}
	I_1&=\left\{0,1,\ldots,\left\lfloor\frac{r}{2}-(2j_3+1)\right\rfloor\right\},\quad I_2=\left\{\left\lfloor\frac{r}{2}-(2j_3+1)\right\rfloor+1,\ldots,\left\lfloor\frac{r}{2}-(2j_2+1)\right\rfloor\right\},\\
	I_3&=\left\{\left\lfloor\frac{r}{2}-(2j_2+1)\right\rfloor+1,\ldots,\left\lfloor\frac{r}{2}-(2j_1+1)\right\rfloor\right\},\quad I_4=\left\{\left\lfloor\frac{r}{2}-(2j_1+1)\right\rfloor+1,\ldots,k_{\max}\right\}.
\end{align*}
The same discussion applies below for other combinations of the inequality relations between $\frac{r}{2}$ and $2j_i+1$. Note that these sets depend on $r$.

Similarly to Lemma \ref{sgnchgE} for the figure-eight knot, the following holds as well.
\begin{lemma}\label{sgnchgB}
The sequence of $A_k(B;j_1,j_2,j_3)$ for $B$ is alternating in $I_1$ and $I_3$ and has constant sign in $I_2$ and $I_4$. 
\end{lemma}
Let $\bm{j}$ denote $(j_1,j_2,j_3)$. We can separate the sum into the alternating parts and the constant sign parts:
\[V_{\bm{j}}^{(r)}(B)=\sum_{k\in I_1}A_k(B;\bm{j})+\sum_{k\in I_2}A_k(B;\bm{j})+\sum_{k\in I_3}A_k(B;\bm{j})+\sum_{k\in I_4}A_k(B;\bm{j}).\]

Now since $k_{\max}$ is equal to $\min\{2j_1,r-(2j_3+1)-1\}$ by the assumption $j_1\leq j_2\leq j_3$, $\frac{2\pi}{r}k_{\max}$ asymptotically tends to $\pi-\frac{\alpha_1}{2}$ or $\pi-\frac{\alpha_3}{2}$ because of comparison between $2j_1$ and $r-(2j_3+1)-1$. Moreover, we may assume that $\alpha_1\leq\alpha_2\leq\alpha_3$; $\min\{\alpha_1,\alpha_2,\alpha_3\}=\alpha_1$.

Because of \eqref{logfact}, we get
\begin{align}
	\log|A_k(B;\bm{j})|&=\frac{r}{2\pi}\left[-\sum_{i=1}^3\left\{\Lambda\left(\frac{2\pi(2j_i+1+k)}{r}\right)-\Lambda\left(\frac{2\pi(2j_i-k)}{r}\right)\right\}\right.\notag\\
&\hspace{10em}\left.-2\Lambda\left(\frac{2\pi k}{r}\right)+2\Lambda\left(\frac{2\pi(2k+1)}{r}\right)\right]+O(\log r)\quad(r\to\infty).\notag\label{summationasymptotics}
\end{align}
Hence, we have 
\begin{equation}
\frac{4\pi}{r}\log|A_k(B;\bm{j})|\sim2\mathrm{Im}(\Phi_B(x))\quad(r\to\infty).\label{coincidenceB}
\end{equation}
Let $T=\tan\theta_0^\pm$, where $\theta_0^\pm$ is the continuous limit of the points satisfying the extremal condition $R_k(B;\bm{j})=\pm 1$ ($r\to\infty$). 
Then we obtain
\[R_k(B;\bm{j})=\frac{(T^2-N_1^2)(T^2-N_2^2)(T^2-N_3^2)}{T^2(1+N_1^2)(1+N_2^2)(1+N_3^2)},\]
where $N_i=\tan\frac{\alpha_i}{2}$ $(i=1,2,3)$. The extremal condition for $B$ gives the two algebraic equations 
\begin{align}
	&T^6-(N_1^2+N_2^2+N_3^2)T^4\notag\\
	&\hspace{2em}+\{N_1^2N_2^2N_3^2+2(N_1^2N_2^2+N_2^2N_3^2+N_3^2N_1^2)+N_1^2+N_2^2+N_3^2+1\}T^2\notag\\
	&\hspace{27em}-N_1^2N_2^2N_3^2=0\label{eqBB}
\end{align}
and
\begin{equation}
	(T^2+1)(T^4-(N_1^2+N_2^2+N_3^2+1)T^2-N_1^2N_2^2N_3^2)=0\label{eqBA},
\end{equation}
where the former (resp. the latter) corresponds to the condition $R_k(B;\bm{j})=-1$ (resp. $R_k(B;\bm{j})=1$). The latter reduces to 
\begin{equation}
T^4-(N_1^2+N_2^2+N_3^2+1)T^2-N_1^2N_2^2N_3^2=0\label{principleB}
\end{equation}
since $T$ is a real number. Moreover, the reduced equation has two real roots corresponding to extremal values. Let $T_B$ be the unique positive root of \eqref{eqBB} satisfying $0<T_B<N_1$, and let $T_A$ be the unique positive root of \eqref{principleB}. Then $T_A>N_3$. In addition, \eqref{eqBB} has either no root between $N_2$ and $N_3$, one double root between $N_2$ and $N_3$, or two distinct roots $N_2<T_B^-<T_B^+<N_3$. 

\begin{lemma}
Assume that $r$ is sufficiently large. 

If \eqref{eqBB} has no root between $N_2$ and $N_3$, then there exist integers $0\leq k_1\leq k_2\leq k_3\leq k_{\max}$ such that $|A_k(B;\bm{j})|$ is non-decreasing on $[0,k_1]$, non-increasing on $[k_1,k_2]$, non-decreasing on $[k_2,k_3]$, and non-increasing on $[k_3,k_{\max}]$. In particular, $|A_k(B;\bm{j})|$ has local maxima at $k=k_1$ and $k=k_3$ and a local minimum at $k=k_2$; $\frac{2\pi k_1}{r}\to\arctan T_B,\frac{2\pi k_2}{r}\to\arctan T_A,\frac{2\pi k_3}{r}\to\pi-\arctan T_A$ when $r\to\infty$. 

If \eqref{eqBB} has two distinct roots $T_B^-,T_B^+$ between $N_2$ and $N_3$, then there exist integers $0\leq k_1\leq k_2\leq k_3\leq k_4\leq k_5\leq k_{\max}$ such that $|A_k(B;\bm{j})|$ is non-decreasing on $[0,k_1]$, non-increasing on $[k_1,k_2]$, non-decreasing on $[k_2,k_3]$, non-increasing on $[k_3,k_4]$, non-decreasing on $[k_4,k_5]$, and non-increasing on $[k_5,k_{\max}]$. In particular, $|A_k(B;\bm{j})|$ has local maxima at $k=k_1,k_3,k_5$ and local minima at $k=k_2,k_4$; $\frac{2\pi k_1}{r}\to\arctan T_B,\frac{2\pi k_2}{r}\to\arctan T_B^-,\frac{2\pi k_3}{r}\to\arctan T_B^+,\frac{2\pi k_4}{r}\to\arctan T_A,\frac{2\pi k_5}{r}\to\pi-\arctan T_A$ when $r\to\infty$. 

If \eqref{eqBB} has a double root $T_B^\ast$ between $N_2$ and $N_3$, then the two possible points converge to $\arctan T_B^\ast$ when they occur.
\end{lemma}
\begin{proof}
To determine the monotonicity and extremal values of $|A_k(B;\bm{j})|$, it suffices to consider the first and second derivatives of $2\,\mathrm{Im}(\Phi_B(x))$. In particular, $\left[\frac{\alpha_1}{2},\frac{\alpha_2}{2}\right]$ and $\left[\frac{\alpha_2}{2},\frac{\alpha_3}{2}\right]$ have no point that yields a value of $2\,\mathrm{Im}(\Phi_B(x))$ greater than $\max\left\{2\,\mathrm{Im}\left(\Phi_B\left(\frac{\alpha_1}{2}\right)\right),0\right\}$. For a maximal (resp. minimal) value of $2\,\mathrm{Im}(\Phi_B(x))$, we know that there exist two points of $|A_k(B;\bm{j})|$ around the point corresponding to the extremal point; we obtain the conclusion by choosing the larger (resp. smaller) one of the two values at the two points.
\end{proof}

\subsection{Proofs of Theorem \ref{MainTheorem2}}\label{PMT2}
In this section, we give a proof of the second main result as follows:

\mainthmS*
\noindent Let $\bm{\alpha}$ denote $(\alpha_1,\alpha_2,\alpha_3)$. The limit formula of Theorem \ref{MainTheorem2} from Conjecture \ref{mainconjecture} is
\[\lim_{\substack{r\to\infty\\r:\mathrm{odd}}}\frac{4\pi}{r}\log|V_{\bm{j}}^{(r)}(B)|=\mathrm{Vol}(M_{\bm{\alpha}}(B)).\]

For the Borromean rings $B$, the proof follows the same strategy as in the case of the figure-eight knot $E$; we first discuss the summands with the single summation index and the extremal condition $|A_k(B;\bm{j})|=|A_{k-1}(B;\bm{j})|$. Recall that the summands $A_k(B;\bm{j})$ are alternating in $I_1$ and $I_3$ and have constant sign in $I_2$ and $I_4$. 

\begin{lemma}\label{constsgn}
It holds that
\begin{equation}
\lim_{\substack{r\to\infty\\r:\mathrm{odd}}}\frac{4\pi}{r}\log\left|\sum_{k\in I_4}A_k(B;\bm{j})\right|=\lim_{\substack{r\to\infty\\r:\mathrm{odd}}}\frac{4\pi}{r}\log\max_{k\in I_4}|A_k(B;\bm{j})|,\label{qesc}
\end{equation}
where $\bm{\alpha}=(\alpha_1,\alpha_2,\alpha_3)$.
\end{lemma}
\begin{proof}
It is clear that
\[\left|\sum_{k\in I_4}A_k(B;\bm{j})\right|\leq\sum_{k\in I_4}|A_k(B;\bm{j})|\leq\#I_4\max_{k\in I_4}|A_k(B;\bm{j})|.\]
Since $A_k(B;\bm{j})$ has a constant sign in $I_4$, we have
\begin{equation}
\left|\sum_{k\in I_4}A_k(B;\bm{j})\right|=\sum_{k\in I_4}|A_k(B;\bm{j})|\geq\max_{k\in I_4}|A_k(B;\bm{j})|.\notag\label{musces}
\end{equation}
Thus, we know
\begin{equation}
\max_{k\in I_4}|A_k(B;\bm{j})|\leq\left|\sum_{k\in I_4}A_k(B;\bm{j})\right|\leq\#I_4\max_{k\in I_4}|A_k(B;\bm{j})|.\label{SqueezeB}
\end{equation}
Hence, from \eqref{SqueezeB}, since $\#I_4=O(r)$ $(r\to\infty)$, we obtain \eqref{qesc}.
\end{proof}

\begin{proposition}[{\cite[Theorem 3.7]{Mednykh2003}}]\label{VolBformula}
The hyperbolic volume of $M_{\alpha_1,\alpha_2,\alpha_3}(B)$ is given by the formula
\begin{equation}
\mathrm{Vol}(M_{\alpha_1,\alpha_2,\alpha_3}(B))=2\left\{\Delta\left(\frac{\alpha_1}{2},\theta_B\right)+\Delta\left(\frac{\alpha_2}{2},\theta_B\right)+\Delta\left(\frac{\alpha_3}{2},\theta_B\right)-2\Delta\left(\frac{\pi}{2},\theta_B\right)-\Delta\left(0,\theta_B\right)\right\},\label{VolB}
\end{equation}
where $\Delta(a,b)=\Lambda(a+b)-\Lambda(a-b)$ and $\theta_B\in\left(0,\frac{\pi}{2}\right)$ is a principal parameter defined by conditions
\begin{align*}
\tan\theta_B&=T,\quad T^4-(N_1^2+N_2^2+N_3^2+1)T^2-N_1^2N_2^2N_3^2=0,\\
N_1&=\tan\frac{\alpha_1}{2},\quad N_2=\tan\frac{\alpha_2}{2},\quad N_3=\tan\frac{\alpha_3}{2}.
\end{align*}
\end{proposition}

\begin{remark}
The cone manifold $M_{\alpha_1,\alpha_2,\alpha_3}(B)$ is hyperbolic when $0\leq\alpha_1,\alpha_2,\alpha_3<\pi$.
\end{remark} 

\begin{remark}
The algebraic equation in the condition of Proposition \ref{VolBformula} is equivalent to \eqref{principleB}.
\end{remark}

From \eqref{coincidenceB}, \eqref{qesc}, and \eqref{VolB}, we know
\begin{equation}
\lim_{\substack{r\to\infty\\r:\mathrm{odd}}}\frac{4\pi}{r}\log\left|\sum_{k\in I_4}A_k(B;\bm{j})\right|=\mathrm{Vol}(M_{\bm{\alpha}}(B)).\label{ItwoVolB}
\end{equation}

\begin{lemma}\label{simplesub}
For $\iota=2,3$, as the odd integer $r$ tends to infinity, $\frac{\left|\sum_{k\in I_\iota}A_k(B;\bm{j})\right|}{\left|\sum_{k\in I_4}A_k(B;\bm{j})\right|}$ tends to zero.
\end{lemma}
\begin{proof}
Assume that $r$ is sufficiently large. Choose a small $\delta>0$ so that $\max\left(2\,\mathrm{Im}\left(\Phi_B\left(\frac{\alpha_1}{2}\right)\right),0\right)+2\delta<\mathrm{Vol}(M_{\bm{\alpha}}(B)).$ We have
\begin{align*}
\left|\sum_{k\in I_\iota}A_k(B;\bm{j})\right|\leq&\sum_{k\in I_\iota}|A_k(B;\bm{j})|\leq\#I_\iota\max_{k\in I_\iota}|A_k(B;\bm{j})|\\
&\leq\#I_\iota\exp\left[\frac{r}{4\pi}\left\{\max\left(2\,\mathrm{Im}\left(\Phi_B\left(\frac{\alpha_1}{2}\right)\right),0\right)+\delta\right\}\right]\quad(\iota=2,3).
\end{align*}
On the other hand, we know
\[\left|\sum_{k\in I_4}A_k(B;\bm{j})\right|\geq\exp\left[\frac{r}{4\pi}\{\mathrm{Vol}(M_{\bm{\alpha}}(B))-\delta\}\right].\]
Therefore,
\begin{align*}
\frac{\left|\sum_{k\in I_\iota}A_k(B;\bm{j})\right|}{\left|\sum_{k\in I_4}A_k(B;\bm{j})\right|}&\leq\#I_\iota\exp\left[\frac{r}{4\pi}\left\{\max\left(2\,\mathrm{Im}\left(\Phi_B\left(\frac{\alpha_1}{2}\right)\right),0\right)+2\delta-\mathrm{Vol}(M_{\bm{\alpha}}(B))\right\}\right]
\end{align*}
tends to zero.
\end{proof}

We first assume that 
$\alpha_i\neq 0$ for $i=1,2,3$. The cases with zero cone angle(s) follow similarly by degeneration (for the complete case, see also \cite{DKY2018}). To prove Theorem \ref{MainTheorem2}, we consider the evaluation of the absolute value of the alternating summation $\left|\sum_{k\in I_1}A_k(B;\bm{j})\right|$. By putting $x_0=\pi-\arctan T_A$, we suppose that $\mathrm{Im}\left(\Phi_B\left(\frac{\alpha_{1}}{2}\right)\right)<\mathrm{Im}\left(\Phi_B\left(x_0\right)\right)$ holds. Choose $U_B$ so that $\max\left\{\mathrm{Im}\left(\Phi_B\left(\frac{\alpha_{1}}{2}\right)\right),0\right\}<U_B<\mathrm{Im}\left(\Phi_B\left(x_0\right)\right)$. We take the interval $I(\bm{\alpha})=\left[0,\frac{\alpha_1}{2}\right]$ and its subintervals $I'(\bm{\alpha})=\left[\varepsilon,\frac{\alpha_1}{2}-\varepsilon\right]$ and $I''(\bm{\alpha})=\left[2\varepsilon,\frac{\alpha_1}{2}-2\varepsilon\right]$ for a small $\varepsilon>0$ such that $\mathrm{Im}(\Phi_B(x))<U_B$ for every $x\in I(\bm{\alpha})\backslash I''(\bm{\alpha})$, where $\bm{\alpha}=(\alpha_1,\alpha_2,\alpha_3)$ is the multi-index of the cone angles. Let $g_{\bm{\alpha}}$ be a smooth function on $\mathbb{R}$ such that $g_{\bm{\alpha}}(t)=0$ if $t$ is outside $I'(\bm{\alpha})$ and $g_{\bm{\alpha}}(t)=1$ if $t$ is in $I''(\bm{\alpha})$. 

Using the quantum dilogarithm function, we can rewrite the summation as
\begin{align*}
\sum_{k\in I_1}A_k(B;\bm{j})
&=\sum_{k\in I_1}\frac{1}{\{1\}}\exp\left[\frac{r}{2\pi\sqrt{-1}}\left\{-\pi\frac{2\pi k}{r}-\frac{2\pi^2}{r}+\sum_{i=1}^3\left(4\frac{2\pi j_i}{r}\frac{2\pi k}{r}+2\frac{2\pi}{r}\frac{2\pi j_i}{r}+2\frac{2\pi}{r}\frac{2\pi k}{r}\right)\right.\right.\\
&\hspace{20em}-3\left(\frac{2\pi k}{r}\right)^2-5\frac{2\pi}{r}\frac{2\pi k}{r}+\frac{4\pi^2}{r^2}\\
&\hspace{13em}+\sum_{i=1}^3\left(-\frac{2\pi\sqrt{-1}}{r}\varphi_r\left(\frac{2\pi(2j_i+1+k)}{r}+\frac{\pi}{r}-b_i\pi\right)\right.\\
&\hspace{17em}+\left.\frac{2\pi\sqrt{-1}}{r}\varphi_r\left(\frac{2\pi(2j_i-k)}{r}+\frac{\pi}{r}-b_i\pi\right)\right)\\
&\hspace{16em}-2\frac{2\pi\sqrt{-1}}{r}\varphi_r\left(\frac{2\pi k}{r}+\frac{\pi}{r}\right)\\
&\hspace{20em}\left.\left.+\,2\frac{2\pi\sqrt{-1}}{r}\varphi_r\left(\frac{2\pi(2k+1)}{r}+\frac{\pi}{r}\right)\right\}\right],
\end{align*}
where each $b_i$ is either 0 or 1, depending on the range of the indices of the Pochhammer symbol. We put 
\begin{align*}
\Phi_B^r\left(\frac{2\pi k}{r}\right)&=-\pi\frac{2\pi k}{r}+4\sum_{i=1}^3\frac{2\pi j_i}{r}\frac{2\pi k}{r}-3\left(\frac{2\pi k}{r}\right)^2+2\frac{2\pi}{r}\sum_{i=1}^3\frac{2\pi j_i}{r}+\frac{2\pi}{r}\frac{2\pi k}{r}-\frac{2\pi^2}{r}+\frac{4\pi^2}{r^2}\\
&\hspace{3em}+\sum_{i=1}^3\left\{-\frac{2\pi\sqrt{-1}}{r}\varphi_r\left(\frac{2\pi(2j_i+1+k)}{r}+\frac{\pi}{r}-b_i\pi\right)\right.\\
&\hspace{8em}\left.+\frac{2\pi\sqrt{-1}}{r}\varphi_r\left(\frac{2\pi(2j_i-k)}{r}+\frac{\pi}{r}-b_i\pi\right)\right\}\\
&\hspace{10em}-2\frac{2\pi\sqrt{-1}}{r}\varphi_r\left(\frac{2\pi k}{r}+\frac{\pi}{r}\right)+2\frac{2\pi\sqrt{-1}}{r}\varphi_r\left(\frac{2\pi(2k+1)}{r}+\frac{\pi}{r}\right).
\end{align*}
Moreover, for a sufficiently large $r$, we regard the sequence of the functions as
\begin{align*}
\Phi_B^r\left(z\right)&=\sum_{i=1}^3\left\{-\frac{2\pi\sqrt{-1}}{r}\varphi_r\left((1-b_i)\pi+\frac{\alpha_i}{2}+z+\frac{3\pi}{r}\right)+\frac{2\pi\sqrt{-1}}{r}\varphi_r\left((1-b_i)\pi+\frac{\alpha_i}{2}-z+\frac{\pi}{r}\right)\right\}\\
&\hspace{4em}-2\frac{2\pi\sqrt{-1}}{r}\varphi_r\left(z+\frac{\pi}{r}\right)+2\frac{2\pi\sqrt{-1}}{r}\varphi_r\left(2z+\frac{3\pi}{r}\right)+(\alpha_1+\alpha_2+\alpha_3+5\pi)z-3z^2.
\end{align*}
Recall that $\frac{2\pi\sqrt{-1}}{r}\varphi_r(z)$ uniformly converges to $\frac{1}{2}\mathrm{Li}_2(e^{2\sqrt{-1}z})$ in a suitable domain; consequently, $\Phi_B^r(z)$ uniformly converges to $\Phi_B(z)$ there. Let $h_{\bm{\alpha},r}(x)=g_{\bm{\alpha}}\left(\frac{2\pi x}{r}\right)\exp\left[\frac{r}{2\pi\sqrt{-1}}\Phi_B^r\left(\frac{2\pi x}{r}\right)\right]$.

\begin{lemma}\label{PoiB}
The function $h_{\bm{\alpha},r}(x)$ is a rapidly decreasing function. Therefore, it holds that
\[\sum_{m\in\mathbb{Z}}h_{\bm{\alpha},r}(m)=\sum_{m\in\mathbb{Z}}\widehat{h_{\bm{\alpha},r}}(m).\]
\end{lemma}

To use the Poisson summation formula by considering the Fourier coefficients
\[\widehat{h_{\bm{\alpha},r}}(m)=\int_\mathbb{R}h_{\bm{\alpha},r}(x)e^{-2\pi mx\sqrt{-1}}dx,\]
we prove the following lemma.
\begin{lemma}\label{evaB}
There exists a constant $M>0$ such that
\[\left|\sum_{m\in\mathbb{Z}}\widehat{h_{\bm{\alpha},r}}(m)\right|\leq Mre^{\frac{r}{2\pi}U_B}\]
for a sufficiently large $r$.
\end{lemma}
\begin{proof}
We similarly have
\begin{align*}
\widehat{h_{\bm{\alpha},r}}(m)&=\frac{1}{2\pi m^2}\int_\mathbb{R}\left[\frac{1}{r}g_{\bm{\alpha}}''(x)+\frac{1}{2\pi\sqrt{-1}}\left\{2g_{\bm{\alpha}}'(x){\Phi_B^r}'(x)+g_{\bm{\alpha}}(x){\Phi_B^r}''(x)+\frac{r}{2\pi\sqrt{-1}}g_{\bm{\alpha}}(x){\Phi_B^r}'(x)^2\right\}\right]\\
&\hspace{23.5em}\times\exp\left[\frac{r}{2\pi\sqrt{-1}}\{\Phi_B^r(x)+2\pi mx\}\right]dx
\end{align*}
for $m\neq 0$ and 
\[\widehat{h_{\bm{\alpha},r}}(0)=\frac{r}{2\pi}\int_{I(\bm{\alpha})}g_{\bm{\alpha}}(x)\exp\left[\frac{r}{2\pi\sqrt{-1}}\Phi_B^r(x)\right]dx.\]
On $I''(\bm{\alpha})$, we have $g_{\bm{\alpha}}(x)=1$ and $g'_{\bm{\alpha}}(x)=g''_{\bm{\alpha}}(x)=0$. Hence, the integral of the core part obtained after the integrations by parts is holomorphic in a neighborhood of the contours considered below. We deform only the core part, while the transition part is kept on the real axis. By the choice of $\varepsilon$, the uniform convergence of $\Phi_B^r$ to $\Phi_B$, and compactness, the transition part is bounded by a polynomial factor in $r$ times $e^{\frac{r}{2\pi}U_B}$. In particular, there also exists a constant $M'>0$ which does not depend on $r$ such that
\begin{align}
|\widehat{h_{\bm{\alpha},r}}(m)|\leq&\,\frac{M'r}{2\pi m^2}\int_{I(\bm{\alpha})}\exp\left[\frac{r}{2\pi}\mathrm{Im}(\Phi_B^r(x)+2\pi mx)\right]dx\quad(m\neq0).\label{nonzeroevB}
\end{align}

Since $\Phi_B^r(z)$ uniformly converges to $\Phi_B(z)$, assume that $r$ is sufficiently large. Moreover, we use a suitable part of a contour of $\mathrm{Im}(\Phi_B(x)+2\pi mx)$ to deform the path of integration. Let $C_{-4}(\bm{\alpha})$ denote the path obtained from $I(\bm{\alpha})$ by keeping $I'(\bm{\alpha})\backslash I''(\bm{\alpha})$ on the real axis and deforming only $I''(\bm{\alpha})$ into the first quadrant so that $\mathrm{Im}(\Phi_B(x) - 8\pi x)<U_B$. Similarly, let $C_{-3}(\bm{\alpha})$ denote the path obtained from $I(\bm{\alpha})$ by keeping $I'(\bm{\alpha})\backslash I''(\bm{\alpha})$ on the real axis and deforming only $I''(\bm{\alpha})$ into the fourth quadrant so that $\mathrm{Im}(\Phi_B(x) - 6\pi x)<U_B$. The paths $C_{-4}(\bm{\alpha})$ and $C_{-3}(\bm{\alpha})$ are shown in Figure \ref{FigPathB}.
\begin{lemma}\label{pathsB}
The paths $C_{-4}(\bm{\alpha})$ and $C_{-3}(\bm{\alpha})$ exist. Moreover, $C_{-4}(\bm{\alpha})$ lies in the first quadrant or on the real axis and $C_{-3}(\bm{\alpha})$ lies in the fourth quadrant or on the real axis.
\end{lemma}
A proof of Lemma \ref{pathsB} is also given in Section \ref{lemmaproofs}.
\begin{figure}[htbp]
\centering\includegraphics[width=7.35cm]{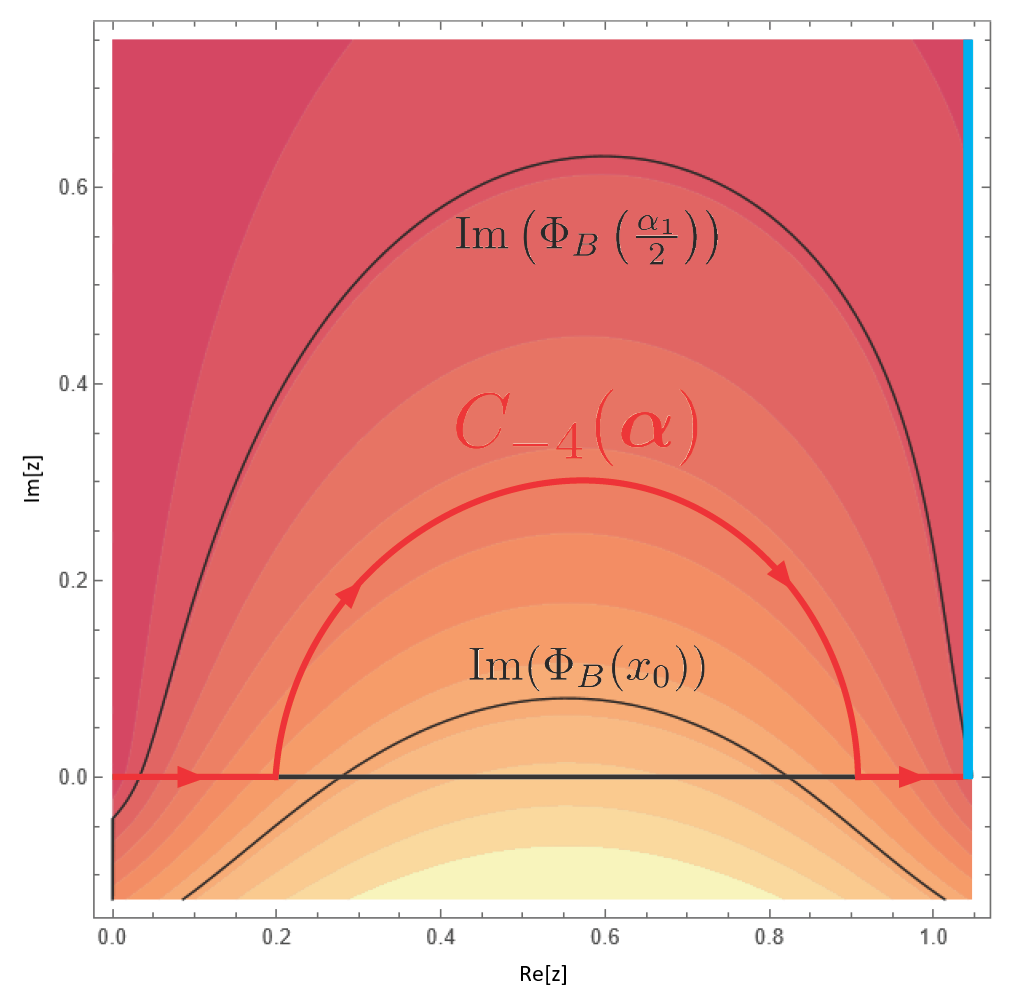}
\centering\includegraphics[width=7.35cm]{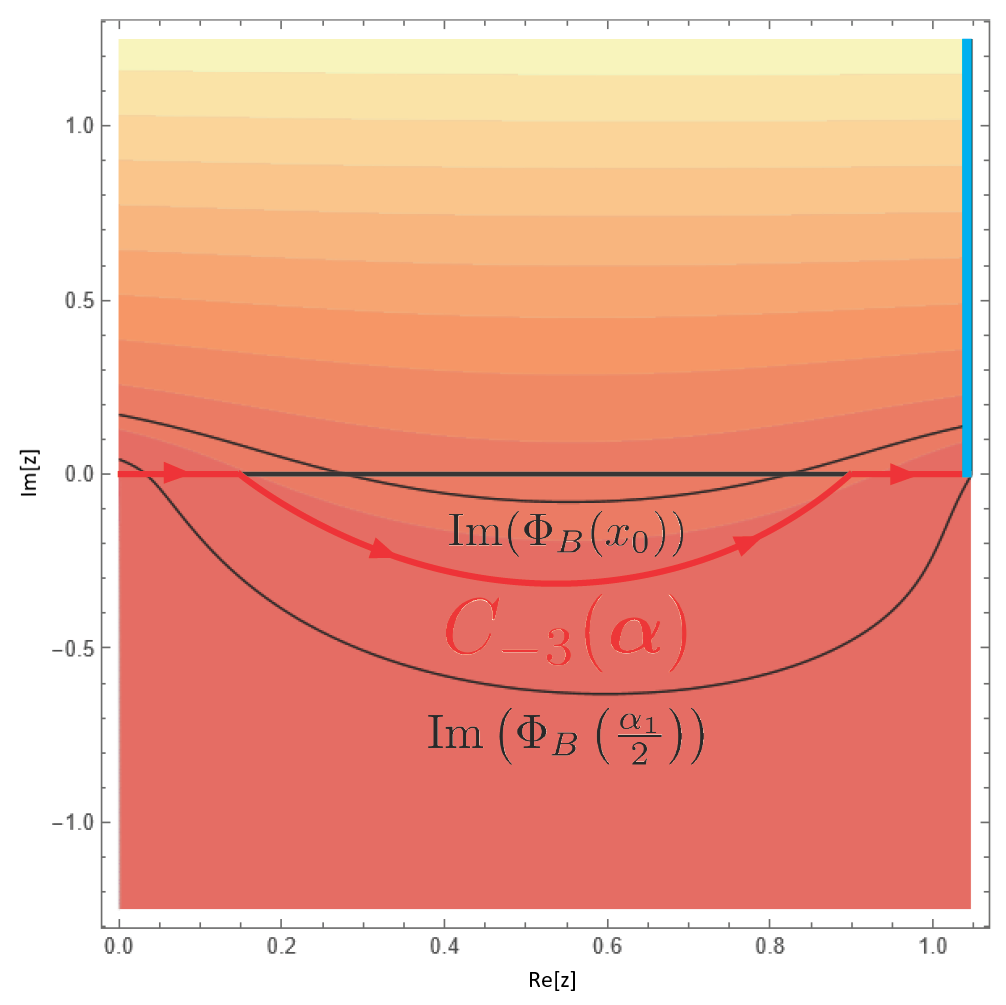}
\caption{The paths of integration along the contours and the real axis: $C_{-4}(\bm{\alpha})$ (left) and $C_{-3}(\bm{\alpha})$ (right) at $\bm{\alpha}=\left(\frac{16\pi}{24},\frac{17\pi}{24},\frac{18\pi}{24}\right)$. The red oriented lines are the integration paths and the blue lines are the branch cuts. The black curves indicate the level sets of $\mathrm{Im}(\Phi_B(x)-8\pi x)$ and $\mathrm{Im}(\Phi_B(x)-6\pi x)$.}\label{FigPathB}
\end{figure}

From \eqref{nonzeroevB} and Lemma \ref{pathsB}, we get
\[|\widehat{h_{\bm{\alpha},r}}(-4)|\leq\frac{M'r}{2\pi\cdot(-4)^2}\int_{C_{-4}({\bm\alpha})}\exp\left[\frac{r}{2\pi}\mathrm{Im}(\Phi_B^r(x)-8\pi x)\right]dx\leq\frac{M'r}{2\pi\cdot(-4)^2}\ell(C_{-4}({\bm\alpha}))e^{\frac{r}{2\pi}U_B}\]
and 
\[|\widehat{h_{\bm{\alpha},r}}(-3)|\leq\frac{M'r}{2\pi\cdot(-3)^2}\int_{C_{-3}({\bm\alpha})}\exp\left[\frac{r}{2\pi}\mathrm{Im}(\Phi_B^r(x)-6\pi x)\right]dx\leq\frac{M'r}{2\pi\cdot(-3)^2}\ell(C_{-3}({\bm\alpha}))e^{\frac{r}{2\pi}U_B}.\]
Furthermore, since $\mathrm{Im}(x)\geq0$ for $x$ in $C_{-4}(\bm{\alpha})$, $\mathrm{Im}(\Phi_B(x)+2\pi mx)\leq\mathrm{Im}(\Phi_B(x)-8\pi x)$ for $m<-4$, we obtain
\begin{equation}
|\widehat{h_{\bm{\alpha},r}}(m)|\leq\frac{M'r}{2\pi m^2}\ell(C_{-4}({\bm\alpha}))e^{\frac{r}{2\pi}U_B}\quad(m<-4).\label{nonzeroB}
\end{equation}
Similarly, since $\mathrm{Im}(x)\leq 0$ in the fourth quadrant or on the real axis, $\mathrm{Im}(\Phi_B(x)+2\pi mx)\leq\mathrm{Im}(\Phi_B(x)-6\pi x)$ for $m>-3$, we also obtain
\begin{equation}
|\widehat{h_{\bm{\alpha},r}}(m)|\leq\frac{M'r}{2\pi m^2}\ell(C_{-3}({\bm\alpha}))e^{\frac{r}{2\pi}U_B}\quad(m>-3,\,m\neq 0).\label{zeroB}
\end{equation}
Note that we also consider the evaluation of $|\widehat{h_{\bm{\alpha},r}}(0)|$ from 
\begin{equation*}
\widehat{h_{\bm{\alpha},r}}(0)=\frac{r}{2\pi}\left\{\int_{I''(\bm{\alpha})}\exp\left[\frac{r}{2\pi\sqrt{-1}}\Phi_B^r(x)\right]dx+\int_{I'(\bm{\alpha})\backslash I''(\bm{\alpha})}g_{\bm{\alpha}}(x)\exp\left[\frac{r}{2\pi\sqrt{-1}}\Phi_B^r(x)\right]dx\right\}
\end{equation*}
similarly.

From \eqref{nonzeroB}, \eqref{zeroB}, and the estimate for $m=0$
, there exists $M>0$ such that 
\[\left|\sum_{m\in\mathbb{Z}}\widehat{h_{\bm{\alpha},r}}(m)\right|\leq Mre^{\frac{r}{2\pi}U_B}.\]
\end{proof}

We now complete the proof of Theorem \ref{MainTheorem2}.
\begin{proof}[Proof of Theorem \ref{MainTheorem2}]
First, 
\[\left|\sum_{k\in I_1}A_k(B;\bm{j})\right|=\frac{1}{|\{1\}|}\left|\sum_{k\in I_1}\exp\left[\frac{r}{2\pi\sqrt{-1}}\Phi_B^r\left(\frac{2\pi k}{r}\right)\right]\right|\]
and $\lim_{\substack{r\to\infty\\r:\mathrm{odd}}}\frac{4\pi}{r}\log\frac{1}{|\{1\}|}=0.$ For a sufficiently large $r$, we have
\begin{align*}
&\left|\sum_{k\in I_1}\exp\left[\frac{r}{2\pi\sqrt{-1}}\Phi_B^r\left(\frac{2\pi k}{r}\right)\right]\right|\\
&=\left|\sum_{m\in\mathbb{Z}}h_{\bm{\alpha},r}(m)+\sum_{k\in \left.I_1\middle\backslash\left(\mathbb{Z}\cap \frac{r}{2\pi}I''(\bm{\alpha})\right)\right.}(1-g_{\bm{\alpha}})\left(\frac{2\pi k}{r}\right)\exp\left[\frac{r}{2\pi\sqrt{-1}}\Phi_B^r\left(\frac{2\pi k}{r}\right)\right]\right|\\
&\leq\left|\sum_{m\in\mathbb{Z}}h_{\bm{\alpha},r}(m)\right|+\sum_{k\in \left.I_1\middle\backslash\left(\mathbb{Z}\cap \frac{r}{2\pi}I''(\bm{\alpha})\right)\right.}\exp\left[\frac{r}{2\pi}\mathrm{Im}\left(\Phi_B^r\left(\frac{2\pi k}{r}\right)\right)\right]\\
&\hspace{1em}\leq\left|\sum_{m\in\mathbb{Z}}h_{\bm{\alpha},r}(m)\right|+\#\left.I_1\middle\backslash\left(\mathbb{Z}\cap \frac{r}{2\pi}I''(\bm{\alpha})\right)\right.e^{\frac{r}{2\pi}U_B}.
\end{align*}
Furthermore, we also have
\[\left|\sum_{m\in\mathbb{Z}}h_{\bm{\alpha},r}(m)\right|=\left|\sum_{m\in\mathbb{Z}}\widehat{h_{\bm{\alpha},r}}(m)\right|\leq Mre^{\frac{r}{2\pi}U_B}\]
from Lemma \ref{PoiB} and Lemma \ref{evaB}. Thus, we know
\begin{equation}
\frac{\left|\sum_{k\in I_1}A_k(B;\bm{j})\right|}{\left|\sum_{k\in I_4}A_k(B;\bm{j})\right|}\leq Mr^2e^{\frac{r}{2\pi}(U_B-\mathrm{Im}(\Phi_B(x_0)))}\to0\quad(r\to\infty,r:\mathrm{odd})\label{expsmallB}
\end{equation}
since we assume that $U_B<\mathrm{Im}\left(\Phi_B\left(x_0\right)\right)$ and Lemma \ref{constsgn} holds. Now the inequality
\begin{align}
&\left|\sum_{k\in I_4}A_k(B;\bm{j})\right|-\left|\sum_{k\in I_1}A_k(B;\bm{j})\right|-\left|\sum_{k\in I_2}A_k(B;\bm{j})\right|-\left|\sum_{k\in I_3}A_k(B;\bm{j})\right|\notag\\
&\hspace{1em}\leq\left|\sum_{k\in I_4}A_k(B;\bm{j})\right|-\left|\sum_{k\in I_1\cup I_2\cup I_3}A_k(B;\bm{j})\right|\notag\\
&\hspace{2em}\leq\left|V_{\bm{j}}^{(r)}(B)\right|\leq\left|\sum_{k\in I_1}A_k(B;\bm{j})\right|+\left|\sum_{k\in I_2}A_k(B;\bm{j})\right|+\left|\sum_{k\in I_3}A_k(B;\bm{j})\right|+\left|\sum_{k\in I_4}A_k(B;\bm{j})\right|\label{triangleB}
\end{align}
holds, and we have
\begin{align}
&\lim_{\substack{r\to\infty\\r:\mathrm{odd}}}\frac{4\pi}{r}\log\left|\left|\sum_{k\in I_4}A_k(B;\bm{j})\right|\pm\left|\sum_{k\in I_1}A_k(B;\bm{j})\right|\pm\left|\sum_{k\in I_2}A_k(B;\bm{j})\right|\pm\left|\sum_{k\in I_3}A_k(B;\bm{j})\right|\right|\notag\\
&\,=\lim_{\substack{r\to\infty\\r:\mathrm{odd}}}\frac{4\pi}{r}\left\{\log\left|\sum_{k\in I_4}A_k(B;\bm{j})\right|\right.\notag\\
&\hspace{6em}\left.+\log\left|1\pm\frac{\left|\sum_{k\in I_1}A_k(B;\bm{j})\right|}{\left|\sum_{k\in I_4}A_k(B;\bm{j})\right|}\pm\frac{\left|\sum_{k\in I_2}A_k(B;\bm{j})\right|}{\left|\sum_{k\in I_4}A_k(B;\bm{j})\right|}\pm\frac{\left|\sum_{k\in I_3}A_k(B;\bm{j})\right|}{\left|\sum_{k\in I_4}A_k(B;\bm{j})\right|}\right|\right\}\notag\\
&\,=\mathrm{Vol}(M_{\bm{\alpha}}(B))\notag\\
&\hspace{3em}+\lim_{\substack{r\to\infty\\r:\mathrm{odd}}}\frac{4\pi}{r}\log\left|1\pm\frac{\left|\sum_{k\in I_1}A_k(B;\bm{j})\right|}{\left|\sum_{k\in I_4}A_k(B;\bm{j})\right|}\pm\frac{\left|\sum_{k\in I_2}A_k(B;\bm{j})\right|}{\left|\sum_{k\in I_4}A_k(B;\bm{j})\right|}\pm\frac{\left|\sum_{k\in I_3}A_k(B;\bm{j})\right|}{\left|\sum_{k\in I_4}A_k(B;\bm{j})\right|}\right|.\label{lastB}
\end{align}
Then we finally obtain
\[\lim_{\substack{r\to\infty\\r:\mathrm{odd}}}\frac{4\pi}{r}\log|V_{\bm{j}}^{(r)}(B)|=\mathrm{Vol}(M_{\bm{\alpha}}(B))\]
from \eqref{expsmallB}, \eqref{triangleB}, \eqref{lastB}, and Lemma \ref{simplesub}.
\end{proof}

\begin{remark}
For example, setting $\alpha=\alpha_1=\alpha_2=\alpha_3$, the condition $\mathrm{Im}\left(\Phi_B\left(\frac{\alpha}{2}\right)\right)<\mathrm{Im}(\Phi_B(x_0))$ is numerically equivalent to $\alpha<2.8225471590\ldots$ by a computation. On the other hand, the boundary of the region of the cone angles satisfying the condition is shown in Figure \ref{valid_region_B}.
\end{remark}

\begin{figure}[htbp]
\centering\includegraphics[width=11cm]{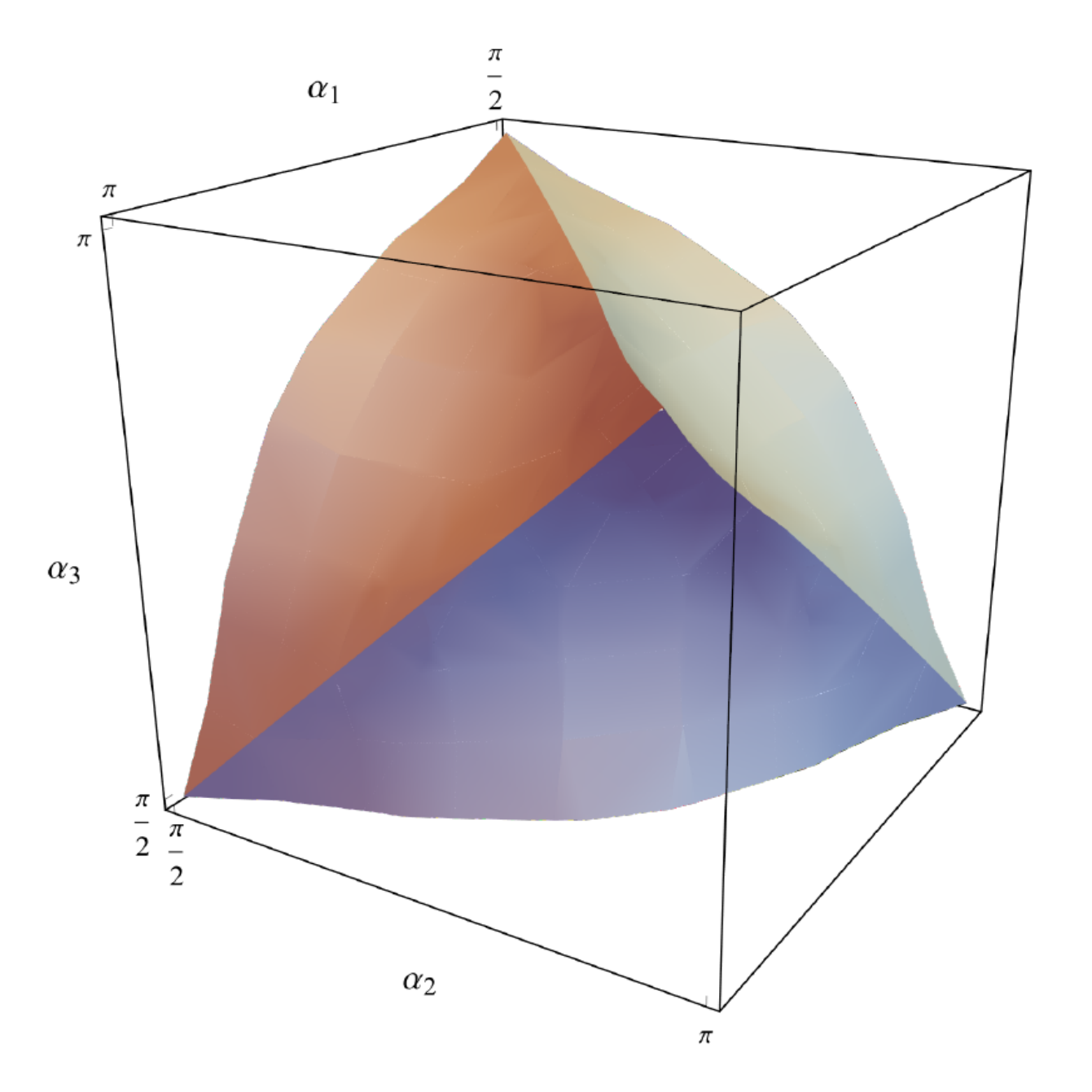}
\caption{A 3-dimensional plot of the boundary of the region of cone angles $\alpha_1, \alpha_2$, and $\alpha_3$ satisfying the condition $\mathrm{Im}\left(\Phi_B\left(\frac{\min\{\alpha_{1},\alpha_2,\alpha_3\}}{2}\right)\right)<\mathrm{Im}(\Phi_B(x_0))$. The axes represent the cone angles in $\left[\frac{\pi}{2}, \pi\right]^3$ for visibility.}\label{valid_region_B}
\end{figure}

\section{Proofs of the lemmata for our main results}\label{lemmaproofs}
In this section, we prove the existence of the contours for evaluations in Section \ref{PMT1} and Section \ref{PMT2}.

First, we prove Lemma \ref{pathsE} for the case of the figure-eight knot $E$. We regard $\mathrm{Im}(\Phi_E(z)+2\pi m z)$ as a real-valued function of two real variables $u$ and $v$, where $z=u+\sqrt{-1}v$. We remark that this function at the origin is always equal to 0. 

\begin{proof}[Proof of Lemma \ref{pathsE}]
Let $f_E(z)=\Phi_E(z)+2\pi mz$. Assume that $0<u<\frac{\alpha}{2}$. Note that $0<\alpha-2u<\frac{2\pi}{3}$ and $0<\alpha+2u<\frac{4\pi}{3}$ away from the real axis along the paths of integration. 

If $w=e^{2\sqrt{-1}z}=e^{-2v}e^{2\sqrt{-1}u}$, then we have
\begin{align*}
\mathrm{Li}_2(e^{\sqrt{-1}\alpha}e^{-2\sqrt{-1}z})&=-\mathrm{Li}_2(e^{-\sqrt{-1}\alpha}e^{2\sqrt{-1}z})-\frac{\pi^2}{6}-\frac{1}{2}\left\{\mathrm{Log}\,(-e^{-\sqrt{-1}\alpha}e^{2\sqrt{-1}z})\right\}^2
\end{align*}
from the inversion formula \eqref{inversion}. Thus, we know
\begin{align*}
f_E(z)&=-\frac{1}{2}\left\{\mathrm{Li}_2(e^{\sqrt{-1}\alpha}w)+\mathrm{Li}_2(e^{-\sqrt{-1}\alpha}w)\right\}+z^2+(2m+1)\pi z+(\text{real constant}),
\end{align*}
where $w=e^{2\sqrt{-1}z}$. In the case where $m=-1$, since
\begin{align*}
\mathrm{Im}(f_E(u+\sqrt{-1}v))|_{m=-1}=(\text{the dilogarithmic part})+(2u-\pi)v
\end{align*}
and $u\leq\frac{\alpha}{2}<\frac{\pi}{3}$, it holds that $\mathrm{Im}(f_E(u+\sqrt{-1}v))|_{m=-1}$ tends to $-\infty$ when $v\to\infty$.

On the other hand, if $w=e^{-2\sqrt{-1}z}=e^{2v}e^{-2\sqrt{-1}u}$, then we have
\begin{align*}
\mathrm{Li}_2(e^{\sqrt{-1}\alpha}e^{2\sqrt{-1}z})&=-\mathrm{Li}_2(e^{-\sqrt{-1}\alpha}e^{-2\sqrt{-1}z})-\frac{\pi^2}{6}-\frac{1}{2}\left\{\mathrm{Log}\,(-e^{-\sqrt{-1}\alpha}e^{-2\sqrt{-1}z})\right\}^2
\end{align*}
using the inversion formula \eqref{inversion} again. Thus, we know
\begin{align*}
f_E(z)&=\frac{1}{2}\left\{\mathrm{Li}_2(e^{\sqrt{-1}\alpha}w)+\mathrm{Li}_2(e^{-\sqrt{-1}\alpha}w)\right\}-z^2+(2m+1)\pi z+(\text{real constant}),
\end{align*}
where $w=e^{-2\sqrt{-1}z}$. In the case where $m=0$, since
\begin{align*}
\mathrm{Im}(f_E(u+\sqrt{-1}v))|_{m=0}=(\text{the dilogarithmic part})+(-2u+\pi)v
\end{align*}
and $u\leq\frac{\alpha}{2}<\frac{\pi}{3}$, it holds that $\mathrm{Im}(f_E(u+\sqrt{-1}v))|_{m=0}$ tends to $-\infty$ when $v\to-\infty$.

The above convergences are uniform with respect to $u\in I''(\alpha)$. Hence, there exist $V_+>0$ and $V_-<0$ such that $\mathrm{Im}(f_E(u+\sqrt{-1}V_+))|_{m=-1}<U_E$ and $\mathrm{Im}(f_E(u+\sqrt{-1}V_-))|_{m=0}<U_E$ for every $u\in I''(\alpha)$. Therefore, the contour defining $C_{-1}(\alpha)$ (resp. $C_{0}(\alpha)$) can be chosen in the first (resp. the fourth) quadrant or on $I(\alpha)$, with $\mathrm{Im}(f_E(u+\sqrt{-1}v))<U_E$.
\end{proof}

Next, we prove Lemma \ref{pathsB} for the case of the Borromean rings $B$. We also regard $\mathrm{Im}(\Phi_B(z)+2\pi mz)$ as a real-valued function of two real variables $u$ and $v$, where $z=u+\sqrt{-1}v$. Note that this function at the origin is always equal to $0$.

\begin{proof}[Proof of Lemma \ref{pathsB}]
Let $f_B(z)=\Phi_B(z)+2\pi mz$. Assume that $0<u<\frac{\alpha_1}{2}$. Note that $0<\alpha_i-2u<\pi$ and $0<\alpha_i+2u<2\pi$ away from the real axis along the paths of integration.

If $w=e^{2\sqrt{-1}z}=e^{-2v}e^{2\sqrt{-1}u}$, then we have
\begin{align*}
\mathrm{Li}_2(e^{\sqrt{-1}\alpha_i}w^{-1})&=-\mathrm{Li}_2(e^{-\sqrt{-1}\alpha_i}w)-\frac{\pi^2}{6}-\frac{1}{2}\left\{\mathrm{Log}\,(-e^{-\sqrt{-1}\alpha_i}e^{2\sqrt{-1}z})\right\}^2
\end{align*}
from the inversion formula \eqref{inversion}. Thus, we know
\begin{align*}
f_B(z)&=-\frac{1}{2}\sum_{i=1}^3\{\mathrm{Li}_2(e^{\sqrt{-1}\alpha_i}w)+\mathrm{Li}_2(e^{-\sqrt{-1}\alpha_i}w)\}-\mathrm{Li}_2(w)+\mathrm{Li}_2(w^2)\\
&\hspace{24em}+8\pi z+2\pi mz+(\text{real constant}),
\end{align*}
where $w=e^{2\sqrt{-1}z}$. Since $v\to\infty$ implies $w\to 0$, $\mathrm{Im}(f_B(u+\sqrt{-1}v))|_{m=-4}$ tends to 0 when $v\to\infty$.

On the other hand, if $w=e^{-2\sqrt{-1}z}=e^{2v}e^{-2\sqrt{-1}u}$, then we have
\begin{align*}
\mathrm{Li}_2(e^{\sqrt{-1}\alpha_i}w^{-1})&=-\mathrm{Li}_2(e^{-\sqrt{-1}\alpha_i}w)-\frac{\pi^2}{6}-\frac{1}{2}\left\{\mathrm{Log}\,(-e^{-\sqrt{-1}\alpha_i}e^{-2\sqrt{-1}z})\right\}^2
\end{align*}
using the inversion formula \eqref{inversion} again. Thus, we know
\begin{align*}
f_B(z)&=\frac{1}{2}\sum_{i=1}^3\{\mathrm{Li}_2(e^{\sqrt{-1}\alpha_i}w)+\mathrm{Li}_2(e^{-\sqrt{-1}\alpha_i}w)\}+\mathrm{Li}_2(w)-\mathrm{Li}_2(w^2)\\
&\hspace{24em}+6\pi z+2\pi mz+(\text{real constant}),
\end{align*}
where $w=e^{-2\sqrt{-1}z}$. Since $v\to-\infty$ implies $w\to 0$, $\mathrm{Im}(f_B(u+\sqrt{-1}v))|_{m=-3}$ tends to 0 when $v\to-\infty$.

The above convergences are uniform with respect to $u\in I''(\bm{\alpha})$. Since $U_B>0$, there exist $V_+>0$ and $V_-<0$ such that $\mathrm{Im}(f_B(u+\sqrt{-1}V_+))|_{m=-4}<U_B$ and $\mathrm{Im}(f_B(u+\sqrt{-1}V_-))|_{m=-3}<U_B$ for every $u\in I''(\bm{\alpha})$. 
Therefore, the contour defining $C_{-4}(\bm{\alpha})$ (resp. $C_{-3}(\bm{\alpha})$) can be chosen in a bounded region of the first (resp. the fourth) quadrant or on $I(\bm{\alpha})$, with $\mathrm{Im}(f_B(u+\sqrt{-1}v))<U_B$.
\end{proof}

\typeout{}
\bibliographystyle{plain}
\bibliography{Ref1}
\end{document}